%% file: main.tex
\documentclass[11pt]{article}
\input{preamble}

\title{Joint estimation of smooth graph signals from partial linear measurements}

\author{ Hemant Tyagi\footnote{This work was supported by a Nanyang Associate Professorship (NAP) grant from NTU Singapore} \\ 
School of Physical and Mathematical Sciences \\ NTU Singapore \\
\texttt{hemant.tyagi@ntu.edu.sg}} 

\date{
\today
} 

\begin{document}
\maketitle

\begin{abstract}
  Given an undirected and connected graph $G$ on $T$ vertices, suppose each vertex $t$ has a latent signal $x_t \in \matR^n$ associated to it. Given partial linear measurements of the signals, for a potentially small subset of the vertices, our goal is to estimate $x_t$'s. Assuming that the signals are smooth w.r.t $G$, in the sense that the quadratic variation of the signals over the graph is small, we obtain non-asymptotic bounds on the mean squared error for jointly recovering $x_t$'s, for the smoothness penalized least squares estimator. In particular, this implies for certain choices of $G$ that this estimator is weakly consistent (as $T \rightarrow \infty$) under potentially very stringent sampling, where only one coordinate is measured per vertex for a vanishingly small fraction of the vertices. 
  The results are extended to a ``multi-layer'' ranking problem where $x_t$ corresponds to the latent strengths of a collection of $n$ items, and noisy pairwise difference measurements are obtained at each ``layer'' $t$ via a measurement graph $G_t$. Weak consistency is established for certain choices of $G$ even when the individual $G_t$'s are very sparse and disconnected.
\end{abstract}


\input{intro}

\input{prob_setup}

\input{mult_layer_transync}

\input{analysis}

\input{exps}

\newpage
\appendix

\input{appendix_analysis} 

\bibliographystyle{plainnat}
\bibliography{references}

\end{document}

%% file: preamble.tex
\usepackage[margin=1in]{geometry}
 
\usepackage[]{amsmath,amssymb,epsfig,natbib}
\usepackage{amsthm}
\usepackage{amsmath}
\usepackage{amssymb}
\usepackage{graphicx}
\usepackage{epstopdf}
\usepackage{comment}
\usepackage{array}
\usepackage{algorithm}
\usepackage{url}

\usepackage{lscape}
\usepackage{algpseudocode}
\usepackage{setspace}
\usepackage{multicol}
\usepackage{multirow}
\usepackage{color}
\usepackage{colortbl}
\usepackage{xcolor}
\usepackage{hyperref}

\usepackage{placeins}

\usepackage[%
    font={small,sf},
    labelfont=bf,
    format=hang,    
    format=plain,
    margin=0pt,
    width=0.8\textwidth,
]{caption}
\usepackage[list=true]{subcaption}

\providecommand{\keywords}[1]
{
  \small	
  \textbf{\textit{Keywords:}} #1
}

\newtheorem{theorem}{Theorem}

\newtheorem{corollary}{Corollary}

\newtheorem{definition}{Definition}

\newtheorem{lemma}{Lemma}

\newtheorem{proposition}{Proposition}
\newtheorem{remark}{Remark}

\numberwithin{equation}{section}

\DeclareMathOperator{\Tr}{Tr}

\DeclareMathOperator{\rank}{rank}

\DeclareMathOperator{\nullsp}{null}
\DeclareMathOperator{\spansp}{span}

\newcommand{\calG}{\ensuremath{\mathcal{G}}}

\newcommand{\calP}{\ensuremath{\mathcal{P}}}

\newcommand{\calS}{\ensuremath{\mathcal{S}}}

\newcommand{\calN}{\ensuremath{\mathcal{N}}}

\newcommand{\calE}{\ensuremath{\mathcal{E}}}

\newcommand{\norm}[1]{\left\|{#1}\right\|}
\newcommand{\abs}[1]{\left|{#1}\right|}

\newcommand{\set}[1]{\left\{{#1}\right\}}
\newcommand{\dotprod}[2]{\left\langle#1,#2\right\rangle}
\newcommand{\est}[1]{\widehat{#1}}
\newcommand{\expec}{\ensuremath{\mathbb{E}}}
\newcommand{\matR}{\ensuremath{\mathbb{R}}}

\newcommand{\argmin}[1]{\underset{#1}{\operatorname{argmin}}}

\newcommand{\prob}{\ensuremath{\mathbb{P}}}

\newcommand{\lambmin}{\ensuremath{\lambda_{\min}}}
\newcommand{\lambmax}{\ensuremath{\lambda_{\max}}}

\newcommand{\lambar}{\ensuremath{\overline{\lambda}}}
\newcommand{\lambarp}{\ensuremath{\overline{\lambda'}}}

\newcommand{\lambminlow}{\ensuremath{\underline{\lambda}_{\min}}}

\newcommand{\lambmaxup}{\ensuremath{\overline{\lambda}_{\max}}}

\newcommand{\ones}{\ensuremath{\mathbf{1}}}

\newcommand{\psum}{p_{\text{sum}}}
\newcommand{\pmax}{p_{\text{max}}}

\newcommand{\Lbar}{\overline{L}}
\newcommand{\Dbar}{\overline{D}}
\newcommand{\Abar}{\overline{A}}


%% file: intro.tex

\section{Introduction} \label{sec:intro}
In this paper, we address the following problem. Consider a collection of latent signals $x_1, x_2,\dots, x_T \in \matR^n$, each of which resides on the vertex of a given undirected graph $G = ([T], \calE)$. Here $\calE$ is the edge-set of $G$. 
For a given a sequence of measurement matrices $C_t \in \matR^{m_t \times n}$, we observe the noisy linear measurements
\begin{equation} \label{eq:partial_meas_model}
   y_t = C_t x_t + \eta_t; \quad t \in [T]  
\end{equation}
where $\eta_t$ denotes noise. Given measurements $(y_t)_{t=1}^T$, the goal is to estimate the latent states $(x_t)_{t=1}^T$. 
The measurements $C_t$ can be ``very limited'' (i.e., $m_t \ll n$), and some $C_t$'s could also be zero. Clearly, barring further assumptions, one cannot hope to do better than the naive strategy of estimating each $x_t$ individually, using only $y_t$.  

In order to perform meaningful estimation, we will assume the node-signals are smooth w.r.t $G$, in the sense that
\begin{equation} \label{eq:smooth_cond_nodesigs}
    \sum_{\set{t,t'} \in \calE} \norm{x_{t} -  x_{t'}}_2^2 \leq S_T
\end{equation}
for some $S_T \geq 0$. Here, the growth of $S_T$ w.r.t $\abs{\calE}$ is important -- while $S_T$ is always at most linear in $\abs{\calE}$ the interesting regime will be when $S_T$ is sub-linear in $\abs{\calE}$. For this regime, we would like to obtain estimates $\est{x}_t$ of $x_t$ for which the mean-squared error (MSE)
\begin{equation*}
    \frac{1}{T}\sum_{t=1}^T \norm{\est{x}_t - x}_2^2 \stackrel{T \rightarrow \infty}{\longrightarrow} 0. \tag{Weak consistency}
\end{equation*}

To this end, we will analyze the following smoothness-penalized least squares estimator
\begin{equation} \label{eq:pen_ls_estimator}
    (\est{x}_t)_{t=1}^T \in \argmin{z_1,\dots z_T \in \matR^n} \left\{\frac{1}{2} \sum_{t=1}^T \norm{y_t - C_t z_t}_2^2 + \frac{\mu}{2} \sum_{\set{t,t'} \in \calE} \norm{z_{t} -  z_{t'}}_2^2 \right\}
\end{equation}
where $\mu \geq 0$ is a regularization parameter. It will be useful to introduce the following notation for clarity. Let $z \in \matR^{nT}$ denote the tall vector formed by column-stacking $z_1,\dots, z_T \in \matR^n$. Defined in an analogous manner, let us introduce the symbols $y \in \matR^{(\sum_{t=1}^T m_t)}$ and $\est{x}, x \in \matR^{n T}$. Let $D \in \set{0,-1,+1}^{\abs{\calE} \times T}$ denote the incidence matrix of $G$ for any orientation of its edges, and 
\begin{equation*}
  M :=  (D \otimes I_n) \in \matR^{(\abs{\calE}n)\times (nT)}, \quad
C := \begin{pmatrix}
C_1 &  & &  \\
& C_2 &  &  \\
& & \ddots &  \\
& & &  & C_T 
\end{pmatrix} \in \matR^{(\sum_{t=1}^T m_t)\times (Tn)}.
\end{equation*}
Then, \eqref{eq:partial_meas_model} is simply a linear regression problem with design matrix $C$ and measurement vector $y$, with the ground truth $x$ satisfying the smoothness condition $\norm{M x}_2^2 \leq S_T$ (which is equivalent to \eqref{eq:smooth_cond_nodesigs}). Moreover, we can rewrite \eqref{eq:pen_ls_estimator} as
\begin{equation} \label{eq:pen_ls_estimator_short}
    \est{x} \in \argmin{z \in \matR^{nT}} \left\{\frac{1}{2}\norm{y - C z}_2^2 + \frac{\mu}{2}\norm{M z}_2^2 \right\}.
\end{equation}
The above problem is motivated by the following lines of work.
\begin{enumerate}
\item \textbf{Signal denoising on graphs.}
In this problem, we typically assume scalar-valued signals residing on the nodes of $G$ (hence $n=1$), and we are given $y = (x + \eta) \in \matR^T$ where $\eta$ denotes noise. In order to meaningfully estimate $x$ from $y$, a common assumption on $x$ is that in \eqref{eq:smooth_cond_nodesigs}, and the estimator in \eqref{eq:pen_ls_estimator_short} with $C \equiv I_T$ -- referred to as Laplacian smoothing (e.g., \citep{SadhanalaTV16}) or Tikhonov regularization (e.g., \citep{shuman13}) -- is a popular choice. The setting in \eqref{eq:partial_meas_model} generalizes this setup with general measurement matrices $C_t$ at each node $t$. This in particular allows us to consider (a) some of the $C_t$'s to be $0$ (subsampling nodes),  and (b) also to obtain partial information for each vector $x_t$ (``subsample'' within a node).

\item \textbf{Multilayer translation synchronization.}
In this problem, $x_t \in \matR^n$ represents the latent strengths of $n$ items, and $C_t$ is the incidence matrix of an (edge-orientation of the) undirected measurement graph $G_t = ([n], \calE_t)$. Thus, $y_t$ corresponds to noisy pairwise differences for a subset of pairs (determined by $\calE_t$) of $x_t$. If $T=1$, the model is known as \emph{translation synchronization} in the literature \citep{huang2017translation}, and consistent recovery of $x_1$ (as $n$ increases) necessarily requires $G_1$ to be connected. In \citep{AKT_dynamicRankRSync}, the general setup with $T \geq 1$ was studied as a model for dynamic ranking from pairwise-comparisons, where the underlying rankings were assumed to evolve smoothly with time (cf., \eqref{eq:smooth_cond_nodesigs}). While $G$ was considered to be the path graph in \citep{AKT_dynamicRankRSync}, with $T$ denoting the number of time points, we will assume $G$ to be any connected graph. Hence, we refer to this problem as \emph{multi-layer}\footnote{This problem is analogous to the multi-layer community detection problem where we seek to combine $T$ sources/layers of information, each corresponding to a graph, with the layers sharing a latent community structure (e.g., \citep{paul2020,coregSC11})} translation synchronization with $T$ layers (or sources) of pairwise information which we seek to meaningfully ``combine''. We are particularly interested in settings where the individual measurement graphs $G_t$ are very sparse and disconnected, so that an individual layer $t$, by itself, is insufficient to meaningfully recover $x_t$. 
\end{enumerate}
As discussed later in Section \ref{subsec:rel_work}, the above problem is also related to the multitask learning problem, which has received significant attention recently (e.g., \citep{nassif_survey20, tian2025learningsimilarlinearrepresentations, du2021fewshotlearninglearningrepresentation}).

%
\subsection{Contributions of the paper}
The main contributions of this work are as follows.
\begin{enumerate}
\item Our main result is Theorem \ref{thm:main_pen_ls} which shows a general error bound on $\norm{\est{x} - x}_2^2$, holding with high probability (w.h.p) when the noise vectors $\eta_t$'s are centered and subgaussian. This result holds even when each $C_t$ is a row-vector, i.e., $m_t = 1$, and only requires that $\sum_{t=1}^T C_t^\top C_t$ is non-singular. Thus, even when $C^\top C$ is itself rank-deficient, we show that the MSE can be bounded meaningfully under the much milder requirement wherein the smallest eigenvalue of $\sum_{t=1}^T C_t^\top C_t$, namely $\lambmin(\sum_{t=1}^T C_t^\top C_t)$, is sufficiently large. The key tool in the analysis is a non-trivial lower bound on the smallest eigenvalue of the matrix $C^\top C + \mu M^\top M$ which holds for any $\mu > 0$ (see Lemma \ref{lem:low_bd_smallest_eig}) and is potentially of independent interest. Roughly, Lemma \ref{lem:low_bd_smallest_eig} states that if $\lambmin(\sum_{t=1}^T C_t^\top C_t) > 0$, then 
\begin{align} \label{eq:lambmin_intro_informal}
\lambmin(C^\top C + \mu M^\top M) \gtrsim
    \begin{cases}
    \mu \alpha(G, C); & \text{when $\mu$ is smaller than a threshold}, \\
    \frac{\lambmin(\sum_{t=1}^T C_t^\top C_t)}{T}; & \text{otherwise},
    \end{cases} 
\end{align}
where the factor $\alpha(G, C) > 0$ depends on the graph $G$ and the measurements $(C_t)_t$.

\item Theorem \ref{thm:main_pen_ls} is instantiated for specific types of $G$, namely for the complete and star graphs, and for a random measurement model (see Proposition \ref{prop:rand_samp_model}) where each $C_t$ is either a uniformly chosen canonical basis vector (with probability $\theta$), or is zero otherwise. For this example setup, we show in Corollary \ref{corr:rand_samp_graphs} that the MSE goes to zero as $T$ increases, provided $S_T = o(\abs{\calE})$ and $\theta = \omega (1/\sqrt{T})$. Thus, consistent recovery is possible even when at most one coordinate is sampled per node, provided $\omega(\sqrt{T})$ nodes are sampled.

\item Finally, we study the multi-layer translation synchronization problem introduced earlier, and outlined formally in Section \ref{sec:disc_multlayer_transsync}. The estimator for $x_t$'s is essentially the same as in \eqref{eq:pen_ls_estimator_short}, but with an additional constraint that enforces the solution $\est{x}_t$ to be centered\footnote{Clearly, each $x_t$ is only identifiable up to a global shift.}, i.e., have $0$-mean. For the ensuing constrained version of \eqref{eq:pen_ls_estimator_short}, we show in Theorem \ref{thm:main_pen_ls_trans_sync} a bound on $\norm{\est{x}-x}_2^2$ holding w.h.p (for subgaussian noise $\eta_t$'s), and only requiring the second smallest eigenvalue of $\sum_{t=1}^T C_t^\top C_t$ to be greater than $0$. This latter condition, crucially, allows for the individual $G_t$'s to be very sparse and disconnected. The proof-steps follow the same top-level ideas as for Theorem \ref{thm:main_pen_ls}, but with differences in the technical details. 

\item Theorem \ref{thm:main_pen_ls_trans_sync} is instantiated for the setup where each $G_t$ is an independently generated Erd\"os-Renyi graph with parameter $p_t$. Corollary \ref{corr:ER_meas_graphs_transync} shows that if $G$ is either the complete or star graph, then the MSE goes to zero as $T$ increases, provided $\sum_{t=1}^T p_t$ grows sufficiently fast w.r.t $T$. Thus meaningful recovery is possible even when each $p_t$ is below the connectivity threshold, i.e., $p_t = o(\log n/n)$, provided $\sum_{t=1}^T p_t$ is large enough.
\end{enumerate}
It is important to point out that when $G$ is the path graph, Theorem \ref{thm:main_pen_ls} only leads to meaningful results when each $C_t^{\top}C_t$ is full-rank (see Remark \ref{rem:prob_with_main_thm}). Similar considerations apply for Theorem \ref{thm:main_pen_ls_trans_sync} in the sense that when $G$ is the path graph, Theorem \ref{thm:main_pen_ls_trans_sync} is meaningful only when each $G_t$ is connected. However, for graphs $G$ for which the Fiedler eigenvalue is sufficiently large (e.g., complete and star graphs), both Theorem's \ref{thm:main_pen_ls} and \ref{thm:main_pen_ls_trans_sync} yield meaningful results in limited-information scenarios where the individual $C_t$'s are highly rank-deficient.

\subsection{Related work} \label{subsec:rel_work}
There are four main lines of work which are closely related to the setting considered in this paper.  
\subsubsection{Denoising smooth signals on a graph} In this setup, it is typically assumed that $n = 1$ and $C_t \equiv 1$ in \eqref{eq:partial_meas_model}, leading to the signal plus noise model
\begin{equation*}
    y = x + \eta; \quad x \in \matR^T,
\end{equation*}
with $\eta$ usually assumed to have independent and centered random entries (Gaussian or subgaussian). When $x$ is smooth w.r.t $G$ in the sense that $\norm{M x}_{2}^2$ is small, and $G$ is a $d$-dimensional grid graph, it was shown in \cite[Theorem's 5 and 6]{SadhanalaTV16} that the estimator \eqref{eq:pen_ls_estimator_short} is minimax optimal in terms of the MSE when $d \in \set{1,2,3}$. \cite{kirichenko2017} studied the denoising problem under a Bayesian regularization framework and assumed that $G$ ``resembles'' a $d$-dimensional grid graph asymptotically as $T \rightarrow \infty$. Under a quadratic smoothness assumption on $x$ similar to that in \eqref{eq:smooth_cond_nodesigs}, and with an appropriate assumption on the prior distribution for randomly generating $x$, they obtained upper bounds on the MSE, and later showed \citep{kirichenko2018} that this rate is optimal.
\cite{pmlr-v130-green21a} recently studied the statistical performance of \eqref{eq:pen_ls_estimator_short} for the setup where $x$ corresponds to i.i.d samples of a smooth $d$-variate function $f_0$ (in the random design setting). Here, $G$ is taken to be a neighborhood graph formed using the design points in $\matR^d$, with the edge weights capturing the similarity of the respective design points. Assuming $f_0$ belongs to the Sobolev smoothness class, it is shown that \eqref{eq:pen_ls_estimator_short} achieves the minimax rate for the MSE for $d=1,2,3$. Moreover, it was also shown that the rates are ``manifold-adaptive'' in the sense that if the domain of $f_0$ is a $m$-dimensional manifold embedded in $\matR^d$, then the MSE rate depends only on $m$ and not $d$.

There exists a long line of work for \emph{semi-supervised learning} on graphs. Here, noisy values of $x$ are given on a subset of the vertices of $G$, and the goal is to estimate the values of $x$ on the unlabeled vertices via \eqref{eq:pen_ls_estimator_short}.  The model in \eqref{eq:partial_meas_model} subsumes this setup as we now have $C_t \in \set{0,1}$ for all $t \in [T]$, with the unlabeled (resp. labeled) vertices corresponding to $t$ for which $C_t = 0$ (resp. $C_t = 1$). In general, there are numerous works focusing on different aspects of this problem  (e.g., \citep{Zhu2003, belkin04, zhouSSL2005, Nadler09}), and which are difficult to cover in full detail. 
\cite{DuSSL2019} studied this problem in a  framework that can be thought of as a semi-supervised version of that in \citep{pmlr-v130-green21a} with $G$ a neighborhood graph (described earlier). For the so-called ``hard estimator'', which involves minimizing $\norm{M z}_2^2$ amongst all $z \in \matR^T$ for which $z_t$ equals $y_t$ for labeled vertices $t$, they show that the estimate $\est{x}_t$ converges to $x_t$ (in probability) as $T \rightarrow \infty$. For the estimator in \eqref{eq:pen_ls_estimator_short}, referred to as the ``soft-estimator'' in \cite{DuSSL2019}, it is shown (not surprisingly) that it is inconsistent on the set of unlabeled vertices (as  $T \rightarrow \infty$) if $\mu$ is chosen to be large enough. A similar hard-estimator was considered in \citep{thorpeSSL19}, but with $G$ now constructed as a random geometric graph (using the design points), and with the objective involving the discrete version of the $p$-Laplacian functional. They estimate the asymptotic behavior of this estimator when the number of unlabeled points increases, with the number of labeled points fixed.

Finally, we remark that another, arguably more popular, smoothness assumption on the signal $x$ is to assume that it has small total variation w.r.t $G$, i.e., $\norm{M x}_1$ is ``small''. This leads to an estimator where the $\norm{Mz}_2^2$ penalty is replaced by $\norm{M z}_1$ in \eqref{eq:pen_ls_estimator_short}. There are many theoretical results in this regard involving bounds on the MSE for recovering $x$, see e.g., \citep{Dalal17, pmlr-v49-huetter16, wang2016trend, mammen1997locally}. In particular, the results of \cite{mammen1997locally} and \cite{wang2016trend} allow for higher order smoothness where $M$ is replaced by $M^k$ for a positive integer $k$ (referred to as trend filtering in the literature).

\subsubsection{Regression for smooth graph signals}
%
There exist results for linear regression problems in a setup more general than ours, where the unknown signal $x$ is smooth w.r.t $G$ \citep{tran2022generalizedelasticnetsquares, Hebiri2011TheSA,LiRaskutti20}. Specifically, in our notation, the model considered therein is 
\begin{equation*}
    y = C x + \eta; \quad x \in \matR^N, \ C \in \matR^{m \times N}
\end{equation*}
where $x$ is assumed to be smooth w.r.t $G = ([N], \calE)$ in the sense of either $\norm{Mx}_2^2$ and/or $\norm{Mx}_1$ being small. Importantly, the design matrix $C$ here is not necessarily block-diagonal as in our setting. However, these results are either not comparable to our results \citep{LiRaskutti20, tran2022generalizedelasticnetsquares}, or lead to pessimistic error bounds \cite{Hebiri2011TheSA}.

\cite{tran2022generalizedelasticnetsquares} consider the setting where the rows of $C$ are independent samples from a centered Gaussian distribution with covariance $\Sigma$. They study a penalized least squares estimator with the penalty $\mu_2 \norm{Mz}_2^2 + \mu_1 \norm{Mz}_1$. For a particular choice of $\mu_1$ (depending on problem parameters) and for any $\mu_2$ satisfying $\mu_2 \lesssim \frac{\mu_1}{\norm{M x}_{\infty}}$ they obtain non-asymptotic bounds on the MSE. However, these bounds essentially capture the smoothness of $x$ via $\norm{M x}_1$, and the term $\norm{Mx}_2^2$ (or bounds thereof) does not appear, unlike our results (e.g. Theorem \ref{thm:main_pen_ls}). Interestingly, the effect of the $\ell_2$ penalization is reflected via the term $\lambmin(\frac{1}{64} \Sigma + \mu_2 M^\top M)$ appearing in the denominator of their error bounds, which is analogous to the quantity $\lambmin(C^\top C + \mu M^\top M)$ in our setting. Note that $C$ in our setting is block-diagonal, hence its rows are necessarily ``non-identical''. Moreover, $C$ is not required to have diagonal blocks which are Gaussian. While no good way of lower bounding $\lambmin(\frac{1}{64} \Sigma + \mu_2 M^\top M)$ was proposed, it was conjectured in \cite[Section 2.2]{tran2022generalizedelasticnetsquares} that $\lambmin(\frac{1}{64} \Sigma + \mu_2 M^\top M) \gtrsim \mu_2$ at least when $\mu_2$ is in a neighborhood of $0$. This is precisely what we obtain for $\lambmin(C^\top C + \mu M^\top M)$, recall \eqref{eq:lambmin_intro_informal}. We believe that the proof outline of Lemma \ref{lem:low_bd_smallest_eig} can be generalized readily to lower bound $\lambmin(\frac{1}{64} \Sigma + \mu_2 M^\top M)$ as well.

The setting of \cite{LiRaskutti20} is similar to that of \cite{tran2022generalizedelasticnetsquares}, but with the following differences: (a) $x$ is also assumed to be sparse leading to an additional penalty on $\norm{z}_1$ in the objective, and (b) $G$ is formed as a weighted graph using the entries of an estimate $\est{\Sigma}$ of the covariance matrix $\Sigma$. Without going into details, the nature of the results are, on a top-level, similar to that of \cite{tran2022generalizedelasticnetsquares}, and the aforementioned considerations apply here as well. In particular, the term $\lambmin(\frac{1}{64} \Sigma + \mu_2 M^\top M)$ appears here as well. It is shown in \cite[Lemma 1]{LiRaskutti20} that for $\mu_2 \in [0,1]$, if the $\ell_1$ norm of each row of $\Sigma$ is lower bounded by $c_l$, then this implies
\begin{align*}
    \lambmin\left(\frac{1}{64} \Sigma + \mu_2 M^\top M \right) \geq (1-\mu_2)\lambmin(\Sigma) + \mu_2 \frac{c_l}{4}.
\end{align*}
The above bound places strong assumptions on $\Sigma$ -- note that in our setting some rows of $C^\top C$ are allowed to be zero, and the parameter $\mu$ can be greater than $1$ as well.

The work of \cite{Hebiri2011TheSA} considers $x$ to be $s$-sparse, and smooth in the sense of $\norm{Mx}_2^2$ being small. They study the penalized least squares estimator with the penalty $\mu_1 \norm{z}_1 + \mu_2\norm{Mz}_2^2$ for a fixed design matrix $C$. The analysis requires the so-called restricted eigenvalue condition to hold for the matrix $\frac{1}{m} C^\top C + \mu_2 M^\top M$ -- this essentially reduces  to $\lambmin(\frac{1}{m} C^\top C + \mu_2 M^\top M)$ when there is no sparsity. Assuming the entries of $\eta$ to be independent Gaussians with variance $\sigma^2$, they show \cite[Corollary 1]{Hebiri2011TheSA} that the estimate $\est{x}$ satisfies (for suitable choices of $\mu_1,\mu_2$) with probability at least $1-\delta$, 
\begin{equation*}
    \norm{\est{x} - x}_1 \lesssim \frac{s \sigma}{ \lambmin(\frac{1}{m} C^\top C + \mu_2 M^\top M)} \sqrt{\frac{\log (N/\delta)}{m}}.
\end{equation*}
In order to compare with our problem setting where $C$ is a block diagonal matrix with each $C_t \in \matR^{1 \times T}$, we set $N = Tn$, $m = T$ and $s = Tn$ (since $x$ is not necessarily sparse). Then the above bound implies
\begin{equation*}
  \frac{1}{T}\norm{\est{x} - x}_2^2 \leq \frac{1}{T}\norm{\est{x} - x}_1^2 \lesssim \frac{n^2 \sigma^2 \log (Tn/\delta)}{\lambmin^2(\frac{1}{m} C^\top C + \mu_2 M^\top M)} .
\end{equation*}
In particular, the above bound on the MSE is weak and does not necessarily imply that the MSE goes to $0$ as $T$ increases.  

%
\subsubsection{Ranking from pairwise comparisons}
As mentioned earlier, one of the motivation behind this work was the recent work of \cite{AKT_dynamicRankRSync}, in the context of dynamic ranking problems. The model in \cite{AKT_dynamicRankRSync} considered $G$ to be a path graph as $T$ referred to the number of time points, and two estimators were proposed for estimating $x_1,\dots,x_T \in \matR^n$ -- one a smoothness penalized estimator similar to \eqref{eq:pen_ls_estimator_short} (with centering constraints), and the other a projection based estimator similar to the Laplacian eigenmaps estimator \cite{SadhanalaTV16}. While MSE bounds were derived which implied weak consistency (as $T \rightarrow \infty$), the analysis required each $G_t$ to be connected as it facilitated obtaining lower bounds on all the eigenvalues of $C^\top C + \mu M^\top M$. While our analysis for the case where $G$ is the path graph also requires each $G_t$ to be connected, we are still able to improve the bounds in \cite[Theorem 2]{AKT_dynamicRankRSync} (see Remark \ref{rem:trans_sync_connect}). Moreover, for other choices of $G$ (e.g. star, complete graphs) we are able to show meaningful error bounds when the individual $G_t$'s are very sparse and disconnected -- this is the main improvement over \cite{AKT_dynamicRankRSync}. 

%
\subsubsection{Multitask learning}
The model \eqref{eq:partial_meas_model} has been considered extensively in the machine learning community, in the context of the so-called multitask learning (MTL) problem, where $T$ refers to the number of ``tasks''. The setting however is different from ours, in the sense that the vectors $x_t$ are typically assumed to belong to an unknown low-dimensional subspace $\calP_t \subset \matR^n$, with each $\calP_t$ assumed to be close to a fixed unknown subspace $\bar{\calP}$. Thus, $x_t$'s are assumed to have similar linear representations, while $x_t$'s themselves can of course be quite different. This was studied in a general setup in \citep{tian2025learningsimilarlinearrepresentations} where a fraction of the tasks are assumed to be corrupted and with $C_t$'s assumed to be random. This work  generalizes other results in this framework where no outliers are allowed \citep{du2021fewshotlearninglearningrepresentation, pmlr-v139-tripuraneni21a, Duanwang23, Chua2021}.

The model \eqref{eq:partial_meas_model} was studied in \cite{nassif2019learning} under the same smoothness assumption \eqref{eq:smooth_cond_nodesigs} as us. They consider a streaming model where for each node $t$, the measurement vectors, i.e. the rows of $C_t$, arrive sequentially. A stochastic-gradient based algorithm is analyzed where the estimate at ``iteration $i$'', namely $\est{x}_t(i)$ , is obtained recursively using the measurement corresponding to the $i$th row of $C_t$. They show that the quantity $\lim_{i \rightarrow \infty} \frac{1}{T} \sum_{t=1}^T \expec \norm{\est{x}_t(i) - x_t}_{2}^2$ is bounded provided the step-size of the iterates is small enough. Other graph-based MTL setups are discussed in the recent survey article \citep{nassif_survey20}.

\subsection{Outline}
Section \ref{sec:prob_setup} provides an overview of the notation, preliminaries needed to carry out the analysis and an overview of the main results, namely Theorem \ref{thm:main_pen_ls}, and its associated corollaries. Section \ref{sec:disc_multlayer_transsync} analyzes the multi-layer translation synchronization problem, and shows how the ideas developed in the proof of Theorem \ref{thm:main_pen_ls} can be applied here. The main result for this problem is outlined as Theorem \ref{thm:main_pen_ls_trans_sync}, while Corollary \ref{corr:ER_meas_graphs_transync} instantiates this for specific choices of $G_t$ and $G$ as described earlier. Section \ref{sec:proof_main_thm_gen} outlines the proof of Theorem \ref{thm:main_pen_ls}, while Section \ref{sec:numerics} contains some simulation results on synthetic examples which validates our theoretical findings.

%% file: prob_setup.tex
\section{Problem setup and main result} \label{sec:prob_setup}
\subsection{Notation} 
For a vector $x \in \matR^n$, $\norm{x}_p$ denotes the $\ell_p$ norm of $x$. For a matrix $X \in \matR^{n \times m}$, the spectral norm of $X$ (i.e., its largest singular value) is denoted by $\norm{X}_2$, while $\dotprod{X}{Y} = \Tr(X^\top Y)$ denotes the inner product between $X$ and $Y$ (here $\Tr(\cdot)$ denotes the trace operator). The column space and null space of $X$ will be denoted by $\spansp(X)$ and $\nullsp(X)$ respectively. 

For symmetric $X \in \matR^{n \times n}$, its  eigenvalues are ordered as 
\begin{equation*}
    \lambda_1(X) \geq \lambda_2(X) \geq \cdots \geq \lambda_n(X).
\end{equation*}
Moreover, we will often use the notation $\lambda_{\min}(X) := \lambda_n(X)$ and $\lambda_{\max}(X) := \lambda_1(X)$. If $X, Y \in \matR^{n \times n}$ are positive semidefinite, recall that the L\"owner ordering $X \succeq Y$ means $X -Y \succeq 0$. This implies $\lambda_i(X) \geq \lambda_i(Y)$ for each $i=1,\dots,n$ (the converse is not necessarily true of course). 

The symbol $\otimes$ denotes the usual Kronecker product between two matrices, $I_n$ refers to the $n \times n$ identity matrix, and $e_1, e_2,\dots, e_n$ denote the canonical basis vectors of $\matR^n$. 
Throughout, we will denote universal constants by $c, c_1, c_2$ etc., whose value may change from line to line. $\log(\cdot)$ denotes the natural logarithm, and $[n] := \set{1,\dots,n}$ for any positive integer $n$. The all ones vector of length $n$ is denoted by $\ones_n$.

For $a,b > 0$, we say $a \lesssim b$ if there exists a constant $c > 0$ such that $a \leq c b$. If $a \lesssim b$ and $a \gtrsim b$, then we write $a \asymp b$.  It will sometimes be convenient to use standard asymptotic notation, namely $O(\cdot), \Omega(\cdot)$, $o(\cdot)$ and $\omega(\cdot)$ (see \cite{algoBook}, for instance).

%
\subsection{Problem setup and preliminaries}

Clearly, $\est{x}$ is a solution of \eqref{eq:pen_ls_estimator_short} iff
\begin{equation} \label{eq:sol_lin_sys}
    \Big(\mu M^\top M + C^\top C \Big) \est{x} = C^\top y.
\end{equation}
The matrix 
\begin{equation*}
    O_T := [C_1^\top \ C_2^\top \ \cdots \ C_T^\top]^{\top} 
         \in \matR^{(\sum_{t=1}^T m_t) \times n}
\end{equation*}
will play a central role in the analysis, and in determining the quality of the estimates. 
It is not difficult to show that \eqref{eq:sol_lin_sys} has a unique solution  under mild conditions wherein $\text{rank} (O_T) = n$, i.e.,  $\lambmin(O_T^\top O_T) > 0$.
%
%
\begin{proposition} \label{prop:uniq_sol_pen_est}
If $\mu > 0$, and $\rank(O_T) = n$, then $\mu M^\top M + C^\top C \succ 0$. 
\end{proposition}
The proof is in Appendix \ref{app:proof_prop_uniq}. Note that even if $m_t = 1$ for all $t=1,\dots,T$, it is easy to construct examples of $(C_t)_t$ such that $\lambmin(O_T^\top O_T) > 0$ holds. For concreteness, it will be instructive to consider the following simple model for $(C_t)_t$, in order to interpret our main result later on.
%
%
\begin{proposition}[Sparse random measurements]\label{prop:rand_samp_model}
 Suppose $C_1,\dots,C_T \in \matR^{1\times n}$ are i.i.d where
 \begin{equation*}
     C_t = \begin{cases}
    0 & \text{with probability $1-\theta$,} \\
    \sim \text{Unif}(\set{e_i^\top}_{i=1}^n) & \text{with probability $\theta$.}
    \end{cases} 
 \end{equation*}
Then for any $\delta \in (0,1)$, if $T \geq \frac{8n}{\theta}\log(\frac{n}{\delta})$, it holds with probability at least $1-2\delta$ that 
\begin{align*}
    \frac{\theta T}{2n} I_n \preceq O_T^\top O_T \preceq \frac{2e\theta T}{n} I_n.
\end{align*}
\end{proposition}
Proposition \ref{prop:rand_samp_model} is a straightforward consequence of the Matrix-Chernoff bounds developed by  \cite{Tropp15MatConc} for sums of independent, random p.s.d matrices. The details are provided in Appendix \ref{appsec:proof_rand_samp_model} for completeness. Observe that this measurement scheme leads to limited information, both within a node (at most one coordinate is sampled) and also across nodes (roughly $\theta T$ nodes are sampled). We will allow $\theta$ to possibly decrease with $T$, as will be seen later.

Denoting the Laplacian of $G$ by $L$, note that $M^\top M = L \otimes I_n$. The eigenvalues of $L$ will be denoted by 
\begin{equation*}
 \lambda_1 \geq \lambda_2 \geq \cdots \geq \lambda_{T-1} > \lambda_T = 0,   
\end{equation*}
where the penultimate inequality follows since $G$ is connected. 
Here, $\lambda_{T-1}$ is the Fiedler eigenvalue of $G$. which is well known to be strictly positive iff $G$ is connected.

Finally, let us recall the definition of subgaussian random vectors.
\begin{definition}[Subgaussian random variables and vectors \citep{HDPbook}]
    A centered random variable $X$ is said to be $\sigma$-subgaussian if 
    \begin{equation*}
        \expec[e^{t X}] \leq e^{\frac{t^2 \sigma^2}{2}}; \quad \forall t\in \matR.
    \end{equation*}
    A random vector $Y \in \matR^n$ is said to be $\sigma$-subgaussian if 
        \begin{equation*}
        \expec[e^{t (u^\top Y)}] \leq e^{\frac{t^2 \sigma^2}{2}}; \quad \forall t\in \matR, \ \forall u \in \matR^n \text{ s.t } \norm{u}_2 = 1.
    \end{equation*}
\end{definition}
Observe that if the entries of the random vector $Y$ are independent and centered $\sigma$-subgaussian random variables, then $Y$ is also $\sigma$-subgaussian. In what follows, the noise vector $\eta \in \matR^{nT}$ -- formed by column-stacking $\eta_t$'s -- will be considered to be $\sigma$-subgaussian.

%
\subsection{Main result}
The following is the main result of this paper. Its  proof is detailed in Section \ref{sec:proof_main_thm_gen}.
\begin{theorem} \label{thm:main_pen_ls}
    Suppose $G$ is connected and $\lambmin(O_T^\top O_T) > 0$. Let $(\eta_t)_{t}$ be centered random vectors such that $\eta \in \matR^{nT}$ -- formed by column-stacking $\eta_t$'s -- is $\sigma$-subgaussian. Assuming $\norm{M x}_2^2 \leq S_T$,  denote $b_1, b_2, b_3$  (all $> 0$) to be the quantities
    \begin{align} \label{eq:b_i_defs}
     \lambda_{T-1} \geq b_1, \quad     \frac{\lambmin(O_T^\top O_T)}{T} \geq b_2, \quad 2 \norm{C}_2\frac{\lambda_{\max}^{1/2}(O_T^\top O_T)}{\sqrt{T}} \leq b_3.
    \end{align}
If $\mu > 0$ then the following is true. 
\begin{enumerate}
    \item Firstly, we have $$\lambmin(\mu M^\top M + C^\top C) \geq \lambar(\mu):= \max\set{\lambarp(\mu), \lambmin(C^\top C)}$$ for $\lambarp(\mu) > 0$ depending only on $\mu$, $b_1, b_2, b_3$ (see Lemma \ref{lem:low_bd_smallest_eig}).  In particular, if $\mu b_1 \geq b_2 + \frac{b_3^2}{b_2}$ then it implies $\lambarp(\mu) \geq \frac{b_2}{4}$.

    \item Additionally, for any $\delta \in (0,e^{-1})$, it holds with probability at least $1-\delta$ that
    \begin{align} \label{eq:main_thm_err_terms} 
    \norm{\est{x} - x}_2^2 &\leq \frac{4\mu}{\lambar(\mu)} S_T + 40 n\sigma^2 \norm{C}_2^2 \left( \sum_{t=1}^{T-1} \frac{1}{(\lambar(\mu) + \mu\lambda_t)^2}  +  \frac{1}{\lambar^2(\mu)}   \right)  \log\Big(\frac{1}{\delta}\Big).
\end{align}
\end{enumerate}
\end{theorem}
The first term in \eqref{eq:main_thm_err_terms} is the bias error which increases with $\mu$, while the second term is the variance error which decreases with $\mu$. 
Note that the bound $\lambmin(\mu M^\top M + C^\top C) \geq  \lambmin(C^\top C)$ is trivially obtained, and is only relevant in the (uninteresting scenario) where each $C_t^\top C_t$ is full-rank. If even one of the matrices $C_t^\top C_t$ were rank-deficient, then we obtain $\lambar(\mu) = \lambarp(\mu)$ as a non-trivial lower bound. 
The term $\lambarp(\mu)$ is non-zero if $\mu > 0$, and the main aspect of the proof is to obtain a closed form expression for this quantity, see Lemma \ref{lem:low_bd_smallest_eig}. Once this is done, the remaining analysis for bounding the bias and variance terms proceeds more or less along standard lines.

The condition that $\eta$ is $\sigma$-subgaussian is satisfied, for instance, if $\eta_{t,i}$'s are i.i.d subgaussian random variables. But this independence assumption is not necessary -- the only requirement on $\eta$ in Theorem \ref{thm:main_pen_ls} is that it is subgaussian. 

Let us now instantiate Theorem \ref{thm:main_pen_ls} for  some specific connected graphs $G$ which will make the dependence on $T$ a bit more explicit. Recall that for the complete graph, $\calE = \set{\set{t,t'}: t\neq t' \in [T]}$. For a star graph with the central node set to $1$ (the choice is arbitrary), we have $\calE = \set{\set{1,t}: 2 \leq t \leq T}$.
%
%
%
\begin{corollary}[Complete graph] \label{corr:compl_graph}
Let $G$ be the complete graph and $\lambminlow(O_T^\top O_T), \lambmaxup(O_T^\top O_T) > 0$ be lower and upper bounds on $\lambmin(O_T^\top O_T), \lambmax(O_T^\top O_T)$ respectively. If $\eta \in \matR^{nT}$ is $\sigma$-subgaussian and $\norm{M x}_2^2 \leq S_T$, then there exist constants $c_1, c_2 > 0$ such that for
\begin{align} \label{eq:mu_compl_graph}
    \mu = \mu^* = \max\Bigg\{&(2n\sigma^2 \norm{C}_2^2)^{1/3}\Big(\frac{\lambminlow(O_T^\top O_T)}{S_T} \Big)^{1/3} \frac{1}{T^{2/3}} - \frac{\lambminlow(O_T^\top O_T)}{T^2}, \nonumber \\
    &\frac{c_1}{T}\Big(\frac{\lambminlow(O_T^\top O_T)}{T} + \norm{C}_2^2\frac{\lambmaxup(O_T^\top O_T)}{\lambminlow(O_T^\top O_T)} \Big)\Bigg\}
\end{align}
we have for any $\delta \in (0,e^{-1})$ that with probability at least $1-\delta$
\begin{align*}
    \norm{\est{x} - x}_2^2 &\leq c_2 \Bigg(\frac{S_T}{T} + \norm{C}_2^2 \frac{\lambmaxup(O_T^\top O_T)}{\lambminlow^2(O_T^\top O_T)} S_T  \\
    &+ \Big(\frac{(n\sigma^2 \norm{C}_2^2)^{1/3} T^{1/3}S_T^{2/3}}{\lambminlow^{2/3}(O_T^\top O_T)} + \frac{n\sigma^2 \norm{C}_2^2 T^2}{\lambminlow^2(O_T^\top O_T)} \Big) \log\Big(\frac{1}{\delta}\Big) \Bigg).
\end{align*}
\end{corollary}

\begin{corollary}[Star graph]\label{corr:star_graph}
Let $G$ be the star graph and $\lambminlow(O_T^\top O_T), \lambmaxup(O_T^\top O_T) > 0$ be lower and upper bounds on $\lambmin(O_T^\top O_T), \lambmax(O_T^\top O_T)$ respectively. If $\eta \in \matR^{nT}$ is $\sigma$-subgaussian and $\norm{M x}_2^2 \leq S_T$, then there exist constants $c_1, c_2 > 0$ such that for
\begin{align} \label{eq:mu_star_graph}
    \mu = \mu^* = \max\Bigg\{&(2n\sigma^2 \norm{C}_2^2)^{1/3}\frac{\lambminlow^{1/3}(O_T^\top O_T)}{S_T^{1/3}} - \frac{\lambminlow(O_T^\top O_T)}{T}, \nonumber \\
    & c_1 \Big(\frac{\lambminlow(O_T^\top O_T)}{T} + \norm{C}_2^2\frac{\lambmaxup(O_T^\top O_T)}{\lambminlow(O_T^\top O_T)} \Big)\Bigg\}
\end{align}
we have for any $\delta \in (0,e^{-1})$ that with probability at least $1-\delta$
\begin{align*}
    \norm{\est{x} - x}_2^2 &\leq c_2 \Bigg( S_T  + \norm{C}_2^2 \frac{\lambmaxup(O_T^\top O_T)}{\lambminlow^2(O_T^\top O_T)} T S_T  \\
    &+ \Big( \frac{(n\sigma^2 \norm{C}_2^2)^{1/3} T S_T^{2/3}}{\lambminlow^{2/3}(O_T^\top O_T)} + \frac{n\sigma^2 \norm{C}_2^2 T^2}{\lambminlow^2(O_T^\top O_T)} \Big) \log\Big(\frac{1}{\delta}\Big) \Bigg).
\end{align*}
\end{corollary}
The proofs of the corollaries are detailed in Appendix \ref{appsec:proofs_corrs_spec_graphs_detsamp}. 
Note that the error decreases as $\lambmin(O_T^\top O_T)$ increases, as we expect, and the statements hold even when $\lambmin(C^\top C) = 0$.  
It is not difficult to think of situations where $\lambmin(O_T^\top O_T)$ (and $\lambmax(O_T^\top O_T)$) both scale linearly w.r.t $T$. For concreteness, consider the situation where $(C_t)_t$ are generated randomly accordingly to the model specified within Proposition \ref{prop:rand_samp_model}. In this case, clearly $\lambmin(C^\top C) = 0$ almost surely. This leads to Corollary \ref{corr:rand_samp_graphs} below, the proof of which follows in a straightforward manner using Corollaries \ref{corr:compl_graph}, \ref{corr:star_graph}, along with Proposition \ref{prop:rand_samp_model}, and is omitted. 
\begin{corollary} \label{corr:rand_samp_graphs}
    Suppose $\eta \in \matR^{nT}$ is $\sigma$-subgaussian and $\norm{M x}_2^2 \leq S_T$. Let $C_1,\dots,C_T \in \matR^{1\times n}$ be i.i.d samples form the distribution specified in Proposition \ref{prop:rand_samp_model}. For $\delta \in (0,e^{-1})$, suppose $T \geq \frac{8n}{\theta}\log(\frac{n}{\delta})$. 
    \begin{enumerate}
        \item If $G$ is the complete graph, then there exist constants $c_1, c_2 > 0$ such that for the choice 
        \begin{align*}
            \mu = \mu^* = \max\set{\sigma^{2/3} \Big(\frac{\theta}{T S_T} \Big)^{1/3} - \frac{\theta}{Tn}, \frac{c_1}{T}}
        \end{align*}
it holds with probability at least $1- 3\delta$, 
\begin{align*}
     \norm{\est{x} - x}_2^2 \leq c_2\left(\frac{n S_T}{\theta T} + \left[ \frac{n\sigma^{2/3} S_T^{2/3}}{\theta^{2/3} T^{1/3}} + \frac{n^3 \sigma^2}{\theta^2} \right] \log\Big(\frac{1}{\delta}\Big)  \right).
\end{align*}

 \item If $G$ is the star graph, then there exist constants $c_1, c_2 > 0$ such that for the choice 
        \begin{align*}
            \mu = \mu^* = \max\set{\sigma^{2/3}  \frac{\theta^{1/3} T^{1/3}}{S_T^{2/3}} - \frac{\theta}{n}, c_1}
        \end{align*}       
        it holds with probability at least $1- 3\delta$, 
\begin{align*}
     \norm{\est{x} - x}_2^2 \leq c_2\left(\frac{n S_T}{\theta} + \left[ \frac{n\sigma^{2/3} T^{1/3} S_T^{2/3}}{\theta^{2/3}} + \frac{n^3 \sigma^2}{\theta^2} \right] \log\Big(\frac{1}{\delta}\Big)  \right).
\end{align*}

\end{enumerate}
\end{corollary}
The above corollary gives us the full picture of how the bounds scale w.r.t the problem parameters (esp. $n, \sigma, T, S_T$) when at most one coordinate is sampled per node. The following remarks are useful to note.
\begin{enumerate}
    \item When $G$ is the complete graph, then the MSE is bounded as 
    \begin{align*}
        \frac{1}{T}\norm{\est{x} - x}_2^2 \lesssim \frac{n S_T}{\theta T^2} + \left[ \frac{n\sigma^{2/3} S_T^{2/3}}{\theta^{2/3} T^{4/3}} + \frac{n^3 \sigma^2}{\theta^2 T} \right] \log\Big(\frac{1}{\delta}\Big).
    \end{align*}
    In particular, for $n, \sigma$ fixed, observe that the MSE is $o(1)$ as $T \rightarrow \infty$, provided
    \begin{align*}
        S_T = o(T^2) \quad \text{and} \quad \theta = \omega\Big(\max\set{\frac{1}{\sqrt{T}}, \frac{S_T}{T^2}}\Big).
    \end{align*}
    This suggests that as the signal $x$ becomes more smooth (i.e., $S_T$ decreases), $\theta$ is allowed to be smaller. If $S_T = o(T^{3/2})$, however, then $\theta = \omega(1/\sqrt{T})$ is needed to ensure that the MSE is $o(1)$ w.r.t. $T$. 

    \item Similarly, when $G$ is the star graph, observe that  
        \begin{align*}
        \frac{1}{T}\norm{\est{x} - x}_2^2 \lesssim \frac{n S_T}{\theta T} + \left[ \frac{n\sigma^{2/3}  S_T^{2/3}}{T^{2/3} \theta^{2/3}} + \frac{n^3 \sigma^2}{\theta^2 T} \right] \log\Big(\frac{1}{\delta}\Big).  
    \end{align*}
    As before, for $n$ and $\sigma$ fixed, the MSE is $o(1)$ when $T \rightarrow \infty$, provided
    \begin{align*}
        S_T = o(T) \quad \text{and} \quad \theta = \omega\Big(\max\set{\frac{1}{\sqrt{T}}, \frac{S_T}{T}}\Big).
    \end{align*}
    If $S_T = o(T^{1/2})$, then $\theta = \omega(1/\sqrt{T})$ is needed to ensure that the MSE is $o(1)$ w.r.t. $T$. 
\end{enumerate}
\begin{remark}[When is Theorem \ref{thm:main_pen_ls} not meaningful?] \label{rem:prob_with_main_thm}
  In the limited information setting where $\lambmin(C^\top C) = 0$, Theorem \ref{thm:main_pen_ls} leads to very pessimistic results for graphs for which the Fiedler eigenvalue $\lambda_{T-1}$ is ``too small''. This is the case, for instance, when $G$ is the path graph where $\calE = \set{\set{t,t+1}: 1 \leq t \leq T-1}$. To see why, consider for concreteness the random measurement setup of Proposition \ref{prop:rand_samp_model}. 
  In this case, we have
  \begin{equation*}
    b_1 = \lambda_{T-1} \asymp \frac{1}{T^2}, \quad b_2 \asymp \frac{\theta}{n} \text{ and } b_3 \asymp \sqrt{\frac{\theta}{n}}
  \end{equation*}
  where we note that $(\lambda_t)_{t=1}^{T-1}$ are known in closed form for path graph's (e.g., \cite{spielmanSGT}). However, we then have the following issues.
  \begin{enumerate}
      \item If $\frac{\mu}{T^2} \gtrsim (\frac{\theta}{n} + 1)$, or equivalently, $\mu \gtrsim T^2$, then $\lambar(\mu) \gtrsim \theta/n$, and the bias error term in \eqref{eq:main_thm_err_terms} is at least of the order $(\frac{n}{\theta}) T^2 S_T$ which is too large. 

      \item On the other hand, if $\mu \lesssim T^2$ then we know from Lemma \ref{lem:low_bd_smallest_eig} that $\lambar(\mu) \gtrsim \frac{\mu}{T^2}$. Consequently, the bias error term in \eqref{eq:main_thm_err_terms} is of the order $T^2 S_T$.
  \end{enumerate}
  \end{remark}

\begin{remark}[When each $C_t$ is full column rank]
Theorem \ref{thm:main_pen_ls} is particularly relevant when at least one $C_t$ has rank less than $n$. In this case, under the conditions of Proposition \ref{prop:uniq_sol_pen_est}, it is not obvious how one may obtain a non-trivial lower bound on $\lambmin(\mu M^\top M + C^\top C)$. If $\rank(C_t) = n$ for each $1 \leq t \leq T$,  we can simply use $\lambar(\mu) \geq \lambmin(C^\top C) \ (> 0)$ in \eqref{eq:main_thm_err_terms}. Then for the case where $G$ is the path graph, one can obtain meaningful bounds on the MSE by following standard calculations involving the Riemannian sum (e.g., \cite{SadhanalaTV16, AKT_dynamicRankRSync}) Whether this is possible even when $\lambmin(C^\top C) = 0$ is left as an open problem.
\end{remark}

%% file: mult_layer_transync.tex
\section{Multi-layer translation synchronization} \label{sec:disc_multlayer_transsync}
We now describe a variation of our problem where each $t \in [T]$ is associated with an undirected ``measurement'' graph $G_t = ([n], \calE_t)$, and $C_t \in \set{-1,1,0}^{\abs{\calE_t} \times n}$ is the incidence matrix of any orientation of the edges of $G_t$. Hence for each $t=1,\dots,T$, we observe the noisy pairwise differences
\begin{equation*}
    (y_t)_{\set{i,j}} = (x_t)_i - (x_t)_j + \eta_{\set{i,j}}; \quad \set{i,j} \in \calE_t.
\end{equation*}
Clearly $x_t$ is identifiable only upto a global-shift, hence we assume each $x_t$ to be centered, i.e., $\ones_n ^\top x_t = 0$. This implies $x \in \matR^{nT}$ is centered, i.e., $x^\top \ones_{nT}= 0$. Denote $P := I_T \otimes (I_n - \frac{1}{n}\ones_n \ones_n^\top)$ to be the ``block-wise'' centering (projection) matrix, then $x = P x$.

As before, $x$ is assumed to be smooth as in \eqref{eq:smooth_cond_nodesigs}, 
thus we estimate $x$ by solving
\begin{equation} \label{eq:pen_ls_estimator_trans_sync}
    \est{x} \in \argmin{\substack{z = Pu, \\ u \in \spansp(P)}} \left\{\frac{1}{2}\norm{y - C z}_2^2 + \frac{\mu}{2}\norm{M z}_2^2 \right\}.
\end{equation}
Plugging $z=Pu$ in the objective, clearly it is minimized by any $u$ satisfying
\begin{equation} \label{eq:lin_sys_transync_u}
    P (\mu M^{\top}M + C^\top C) P u = P C^\top y.
\end{equation}
Along similar lines as in Proposition \ref{prop:uniq_sol_pen_est}, it is easy to show that \eqref{eq:pen_ls_estimator_trans_sync} has a unique solution 
under mild conditions. All the proofs from this section are outlined in Appendix \ref{appsec:proofs_trans_sync}.
%
\begin{proposition} \label{prop:uniq_sol_est_trans_sync}
If $\mu > 0$, and $\rank (O_T) = n-1$, then $\nullsp(\mu M^{\top}M + C^\top C) = \spansp(\ones_{nT})$. Consequently, $\spansp(P) \subseteq \spansp(\mu M^{\top}M + C^\top C)$.
\end{proposition}
Observe that $\nullsp(O_T)$ always contains $\spansp(\ones_n)$ since $C_t$ is an incidence matrix. Hence for $\mu > 0$, Proposition \ref{prop:uniq_sol_est_trans_sync} implies that if $\nullsp(O_T) = \spansp(\ones_n)$, or equivalently $\lambda_{n-1}(O_T^\top O_T) > 0$, then $P (\mu M^{\top}M + C^\top C) P$ has the same column-span as $P$. Consequently, \eqref{eq:lin_sys_transync_u} has a unique solution $u \in \spansp(P)$, and so, the solution of \eqref{eq:pen_ls_estimator_trans_sync} is uniquely given by 
\begin{equation} \label{eq:cen_pen_estimator_trans_sync}
    \est{x} = P \Big(P (\mu M^{\top}M + C^\top C) P \Big)^{\dagger}P C^\top y = \Big(P (\mu M^{\top}M + C^\top C) P \Big)^{\dagger} C^\top y.
\end{equation}
It turns out that we can bound $\norm{\est{x} - x}_2^2$ by proceeding in the same manner as Theorem \ref{thm:main_pen_ls}. 
%
\begin{theorem} \label{thm:main_pen_ls_trans_sync}
    Suppose $G$ is connected and let $C_t \in \set{-1,1,0}^{\abs{\calE_t} \times n}$ be the incidence matrix of any orientation of the edges of $G_t$ (for each $t$). Suppose $\lambda_{n-1}(O_T^\top O_T) > 0$ and let $(\eta_t)_{t}$ be centered random vectors such that $\eta \in \matR^{nT}$ -- formed by column-stacking $\eta_t$'s -- is $\sigma$-subgaussian. Assume $\norm{M x}_2^2 \leq S_T$ and each $x_t$ is centered. Let $b_1,b_3$ be as defined in \eqref{eq:b_i_defs}, and denote $b_2$ as
    \begin{align} \label{eq:b2_def_transync}
          \frac{\lambda_{n-1}(O_T^\top O_T)}{T} \geq b_2.
    \end{align}
If $\mu > 0$ then the following is true. 
\begin{enumerate}
    \item Firstly, we have $$\lambda_{T(n-1)}\Big(P(\mu M^\top M + C^\top C)P \Big) \geq \lambar(\mu) := \max\set{\lambarp(\mu), \lambda_{T(n-1)}(C^\top C)} $$ for $\lambarp(\mu) > 0$ depending only on $\mu$, $b_1, b_2, b_3$ as in Lemma \ref{lem:low_bd_smallest_eig}.   In particular, if $\mu b_1 \geq b_2 + \frac{b_3^2}{b_2}$ then it implies $\lambarp(\mu) \geq \frac{b_2}{4}$.

    \item Consider the estimate $\est{x}$ of $x$ in \eqref{eq:cen_pen_estimator_trans_sync}. For any $\delta \in (0,e^{-1})$, $\norm{\est{x}-x}_2^2$ satisfies the bound in \eqref{eq:main_thm_err_terms} for $\lambar(\mu)$ as defined above, with probability at least $1-\delta$. 
\end{enumerate}
\end{theorem}
In contrast to Theorem \ref{thm:main_pen_ls},  Theorem \ref{thm:main_pen_ls_trans_sync} requires $\lambda_{n-1}(O_T) = \lambda_{n-1}(\sum_{t=1}^T C_t^\top C_t)$ to be sufficiently large. Its proof follows the same steps as that of Theorem \ref{thm:main_pen_ls}, but also inherits the same deficiencies when $\lambda_{T(n-1)}(C^\top C) = 0$ -- the bounds are meaningful for graphs $G$ which are sufficiently connected (cf. Remark \ref{rem:prob_with_main_thm}). Nevertheless, if $G$ is sufficiently connected, we obtain meaningful error bounds, even when each $G_t$ is very sparse and  disconnected (leading to $\lambda_{T(n-1)}(C^\top C) = 0$).  
\paragraph{Erd\"os-Renyi measurement graphs.} Now consider the setting where $(G_t)_t$ is a sequence of independent Erd\"os-Renyi graphs where $G_t \stackrel{\text{ind.}}{\sim} \calG(n, p_t)$. Hence each edge of $G_t$ is an independent Bernoulli random variable with parameter $p_t$. In this setting, we can show that $O_T^\top O_T$ is well-conditioned provided $\sum_{t=1}^T p_t$ is sufficiently large.
\begin{proposition}\label{prop:erdos_renyi_meas_graph}
Suppose $G_t \stackrel{\text{ind.}}{\sim} \calG(n, p_t)$ are independent Erd\"os-Renyi graphs. Denote $\psum:= \sum_{t=1}^T p_t$ and $\pmax := \max_t p_t$. For any $\delta \in (0,1)$, the following holds.
\begin{enumerate}
    \item There exists a constant $c > 0$ such that if $\psum \geq c \frac{\log(n/\delta)}{n}$, then
    $$\prob\Big(\frac{n\psum}{2} \leq \lambda_{n-1}(O_T^\top O_T) \leq \lambmax(O_T^\top O_T) \leq \frac{3n\psum}{2} \Big) \geq 1-\delta.$$  

    \item $\prob\Bigg(\norm{C}_2 \leq \underbrace{\min \set{\sqrt{2n\pmax} + \Big(2n \log\Big(\frac{nT}{\delta}\Big)\Big)^{1/4},  \sqrt{2n}}}_{:=\gamma_{n,T}} \Bigg) \geq 1-\delta$.
\end{enumerate}
\end{proposition}
The proof of the first part uses a concentration result from \cite[Theorem 1]{paul_ml_sbm20} for the sum of independent adjacency matrices of inhomogeneous\footnote{This means each edge $\set{i,j}$ is sampled from a Bernoulli distribution with parameter $p_{ij}$.} Erd\"os-Renyi graphs. The second part follows using Gershgorin's theorem along with Hoeffding's inequality. With this in hand, we now obtain the following corollary of Theorem \ref{thm:main_pen_ls_trans_sync} when $G$ is the complete graph, or a star graph. The proof is a direct consequence of Corollaries \ref{corr:compl_graph} and \ref{corr:star_graph}, along with Proposition \ref{prop:erdos_renyi_meas_graph}. Indeed, we can simply replace $\lambminlow(O_T^\top O_T)$ with a lower bound on $\lambda_{n-1}(O_T^\top O_T)$, and also replace $\norm{C}_2$ with its upper bound. 
\begin{corollary} \label{corr:ER_meas_graphs_transync}
    Suppose $\eta \in \matR^{nT}$ is $\sigma$-subgaussian, $\norm{M x}_2^2 \leq S_T$, and each $x_t$ is centered. Moreover, let $G_t \stackrel{\text{ind.}}{\sim} \calG(n, p_t)$ be independent Erd\"os-Renyi graphs. Recall $\psum, \pmax$ and $\gamma_{n,T}$ from Proposition \ref{prop:erdos_renyi_meas_graph}. 
    There exist constants $c_1,c_2,c_3 > 0$ such that for any $\delta \in (0,e^{-1})$, if $\psum \geq c_1 \frac{\log(n/\delta)}{n}$ then the following is true for the estimate $\est{x}$ in \eqref{eq:cen_pen_estimator_trans_sync}.
    \begin{enumerate}
        \item If $G$ is the complete graph, then for the choice 
        \begin{align*}
            \mu = \mu^* = \max\set{n^{2/3}\sigma^{2/3} \gamma^{2/3}_{n,T} \Big(\frac{\psum}{T^2 S_T} \Big)^{1/3} - \frac{n \psum}{T^2}, \frac{c_2}{T} \left(\frac{n\psum}{T} + \gamma^{2}_{n,T} \right)}
        \end{align*}
it holds with probability at least $1- 3\delta$, 
\begin{align*}
     \norm{\est{x} - x}_2^2 &\leq c_3\left(\frac{ S_T}{T} +  \frac{\gamma^{2}_{n,T}}{n\psum} S_T + 
     \Big[\frac{\sigma^{2/3} \gamma^{2/3}_{n,T} T^{1/3} S_T^{2/3}}{n^{1/3}\psum^{2/3}} + \frac{\sigma^2 \gamma^{2}_{n,T} T^2}{n\psum^2} \Big] \log\Big(\frac{1}{\delta}\Big)  \right).
\end{align*}

 \item If $G$ is the star graph, then for the choice 
        \begin{align*}
            \mu = \mu^* = \max\set{n^{2/3}\sigma^{2/3} \gamma^{2/3}_{n,T} \Big(\frac{\psum}{S_T} \Big)^{1/3} - \frac{n \psum}{T}, \frac{c_2}{T} \left(\frac{n\psum}{T} + \gamma^{2}_{n,T} \right)}
        \end{align*}   
        it holds with probability at least $1- 3\delta$, 
\begin{align*}
     \norm{\est{x} - x}_2^2 &\leq c_3\left(S_T +  \frac{\gamma^{2}_{n,T}}{n\psum} T S_T + 
     \Big[\frac{\sigma^{2/3} \gamma^{2/3}_{n,T} T S_T^{2/3}}{n^{1/3}\psum^{2/3}} + \frac{\sigma^2 \gamma^{2}_{n,T} T^2}{n\psum^2} \Big] \log\Big(\frac{1}{\delta}\Big)  \right).
\end{align*}
\end{enumerate}
\end{corollary}
The key quantity in the above bounds is $\psum$ which in turn arises from the lower bound on $\lambda_{n-1}(O_T^\top O_T)$ in Proposition \ref{prop:erdos_renyi_meas_graph}. While the statement of Corollary \ref{corr:ER_meas_graphs_transync} holds under mild requirements on $\psum (\gtrsim \frac{\log (n/\delta)}{n})$, note that meaningful error bounds are obtained when $\psum$ grows with $T$ at a sufficiently fast rate. To see this, suppose $\sigma$ and $n$ are fixed.
\begin{enumerate}
    \item Then, when $G$ is the complete graph, we see that the $\text{MSE} = o(1)$ as $T$ increases, provided
    \begin{equation*}
        S_T = o(T^2) \ \text{ and } \ \psum= \omega\left(\max\set{\frac{S_T}{T}, \sqrt{T}} \right).
    \end{equation*}

    \item When $G$ is the star graph, we see that the $\text{MSE} = o(1)$ as $T$ increases, provided
    \begin{equation*}
        S_T = o(T) \ \text{ and } \ \psum= \omega\left(\max\set{S_T, \sqrt{T}} \right).
    \end{equation*}
\end{enumerate}

\begin{remark}[When each $G_t$ is connected] \label{rem:trans_sync_connect}
 In case each $G_t$ is connected, i.e. $\rank(C_t) = n - 1$ for each $t$, then denoting $\beta := \min_t \lambda_{n-1}(C_t^\top C_t) $, we have $\lambda_{T(n-1)}(C^\top C) = \beta$. Invoking Theorem \ref{thm:main_pen_ls_trans_sync}, we can then use $\lambar(\mu) \geq \beta$ in \eqref{eq:main_thm_err_terms}. 
 In particular, if $G$ is the path graph, 
 then by following the calculations in \cite[Proof of Lemma 2]{AKT_dynamicRankRSync}, one can show (for fixed $\sigma, n$) that  
\begin{equation*}
    \norm{\est{x} - x}_2^2 = O(\mu S_T + T/\sqrt{\mu}) = O(T^{2/3} S_T^{1/3})
\end{equation*}
for the optimal choice $\mu \asymp (T/S_T)^{2/3}$.  
 This matches the optimal $\ell_2$ error rate for estimating smooth signals on a path graph in the sequence model \cite[Theorems 5,6]{SadhanalaTV16}. It also improves upon \cite[Theorem 2]{AKT_dynamicRankRSync}. Indeed, the bias term in \cite[Theorem 2]{AKT_dynamicRankRSync} scaled suboptimally with $\mu$ as $O(\mu^2)$ while we obtain the scaling $O(\mu)$.
\end{remark}

%% file: analysis.tex
\section{Proof of Theorem \ref{thm:main_pen_ls}} \label{sec:proof_main_thm_gen}
In Section \ref{subsec:low_bd_eig_val} (Lemma \ref{lem:low_bd_smallest_eig}) we will show a lower bound on the smallest eigenvalue of $\mu M^\top M +  C^\top C$, which is also the main part of the analysis. Then, we will proceed to bound the estimation error by bounding the bias (see Lemma \ref{lem:bias_err}) and variance terms (see Lemma \ref{lem:var_error}). This is shown in Section \ref{subsec:est_err_bd}, and proceeds more or less  along standard lines. This completes the proof of Theorem \ref{thm:main_pen_ls}. Section \ref{sec:bias_var_err} contains the proofs of Lemmas \ref{lem:bias_err} and \ref{lem:var_error}.
%
%
\subsection{Lower bounding the smallest eigenvalue} \label{subsec:low_bd_eig_val}
Let us denote 
\begin{equation} \label{eq:var_form_small_eig}
    \lambmin(\mu M^\top M +   C^\top C) = \min_{\norm{u}_2^2 = 1} u^\top (\mu M^\top M + C^\top C) u
\end{equation}  
to be the smallest eigenvalue of $\mu M^\top M +  C^\top C$. 
We will prove the following lemma.
%
%
\begin{lemma} \label{lem:low_bd_smallest_eig}
   For $b_1, b_2$ and $b_3$ as in \eqref{eq:b_i_defs}, we have $$\lambmin(\mu M^\top M + C^\top C) \geq \lambar(\mu) := \max\set{\lambarp(\mu), \lambmin(C^\top C)}$$ where
    \begin{align*}
        \lambarp(\mu) 
        = 
    \begin{cases}
    \frac{b_2^2}{2(b_2^2 + b_3^2)}\mu b_1; & \text{if $\mu b_1 < b_2 + \frac{b_3^2 - b_2^2}{2b_2}$,} \\
    \frac{1}{4} \left(1-\frac{\mu b_1 - b_2}{\sqrt{b_3^2 + (\mu b_1 - b_2)^2}} \right)\mu b_1 + \frac{b_2}{4} + \left(\frac{b_2(\mu b_1 - b_2) - b_3^2}{4\sqrt{b_3^2 + (\mu b_1 - b_2)^2}} \right);
    & \text{if $\mu b_1 \geq b_2 + \frac{b_3^2 - b_2^2}{2b_2}$.}
    \end{cases} 
    \end{align*}
   In particular, if $\mu b_1 \geq b_2 + \frac{b_3^2}{b_2}$ then it implies $\lambarp(\mu) \geq \frac{b_2}{4}$.
\end{lemma}
\begin{proof}
Since $\lambmin(\mu M^\top M + C^\top C) \geq \lambmin(C^\top C)$ is trivially verified, we only focus on showing that $\lambmin(\mu M^\top M + C^\top C) \geq \lambarp(\mu)$ holds as well.

By writing $u = u_1 + u_2$ where $u_1 \in \nullsp(M^\top M)$ and $u_2 \in \nullsp^{\perp}(M^\top M)$, we obtain
\begin{align} 
    &u^\top (\mu M^\top M + C^\top C) u 
    \nonumber \\
    &\geq \max\set{u^\top (\mu M^\top M) u, u^\top(C^\top C) u} \nonumber \\
    &\geq \max\set{\mu u_2^\top  M^\top M u_2,  u_1^\top (C^\top C)u_1 +  2u_1^\top (C^\top C) u_2} \tag{since $u_2^\top (C^\top C) u_2 \geq 0$} \nonumber \\
    &\geq \max\set{\mu u_2^\top  M^\top M u_2,  u_1^\top (C^\top C)u_1 - \abs{2u_1^\top (C^\top C) u_2}}. \label{eq:var_form_ineq}
\end{align}
Now consider for $\alpha \in [0,1]$ the set 
\begin{align*}
    \calS_{\alpha} := \set{u=u_1+u_2: u_1 \in \nullsp(M^\top M), u_2 \in \nullsp^{\perp}(M^\top M), \norm{u_1}_2^2 = \alpha, \norm{u_2}_2^2 = 1-\alpha}.
\end{align*}
Clearly, we can rewrite \eqref{eq:var_form_small_eig} as
\begin{equation} \label{eq:var_form_small_eig_new}
    \lambmin(\mu M^\top M + C^\top C) = \min_{\alpha \in [0,1]} \min_{u \in \calS_{\alpha}} u^\top (\mu M^\top M + C^\top C) u.
\end{equation}

For any $u \in \calS_{\alpha}$ the following holds.
\begin{itemize}
    \item Since $u_2 \in \nullsp^{\perp}(M^\top M)$, we have 
    \begin{align*}
        \mu u_2^\top (M^\top M) u_2 \geq \mu(1-\alpha) \lambda_{n(T-1)}(M^\top M) = \mu(1-\alpha) \lambda_{T-1} \geq \mu(1-\alpha) b_1 \tag{recall \eqref{eq:b_i_defs}}
    \end{align*}
    
    where we used the fact $\lambda_{n(T-1)}(M^\top M) = \lambda_{T-1}$ since $M^\top M = L \otimes I_n$.

    \item Any $u_1 \in \nullsp(M^\top M)$ (with $\norm{u_1}_2^2 = \alpha$) is of the form
    \begin{align} \label{eq:u1_form}
        u_1 = \frac{\sqrt{\alpha}}{\sqrt{T}} \begin{pmatrix}
I \\
I  \\
\vdots \\
I
\end{pmatrix} \frac{w}{\norm{w}_2}; \quad w (\neq 0) \in \matR^n
\end{align}
which implies
\begin{equation*}
    u_1^\top (C^\top C)u_1 = \frac{\alpha}{T\norm{w}_2^2} w^\top \left(\sum_{i=1}^T (C_i)^\top C_i \right) w \geq \alpha \frac{\lambmin(O_T^\top O_T)}{T} \geq \alpha b_2 \tag{recall \eqref{eq:b_i_defs}}.
\end{equation*}

\item Finally, we can upper bound 
\begin{align*}
    &\abs{2u_1^\top (C^\top C) u_2} \\ 
    &\leq 2\norm{u_2}_2 \norm{u_1^\top (C^\top C)}_2\\
    &\leq 2\sqrt{1-\alpha}\frac{\sqrt{\alpha}}{\sqrt{T} \norm{w}_2} \norm{w^\top [I \ I  \ \cdots \ I] (C^\top C)}_2 \tag{using \eqref{eq:u1_form}}\\
    &\leq \frac{2\sqrt{\alpha(1-\alpha)}}{\sqrt{T}} \norm{C}_2 \frac{\lambda_{\max}^{1/2}(O_T^\top O_T)}{\sqrt{T}} \\
    &\leq \sqrt{\alpha(1-\alpha)} b_3. \tag{recall \eqref{eq:b_i_defs}}
\end{align*}
\end{itemize}
Using the above bounds with \eqref{eq:var_form_ineq} and \eqref{eq:var_form_small_eig_new} leads to
\begin{align}
    \lambmin(\mu M^\top M + C^\top C) 
    \geq \min_{\alpha \in [0,1]} \max \set{\underbrace{(1-\alpha)\mu\lambda_{nT}(M^{\top}M)}_{=: D_1(\alpha)}, \underbrace{\alpha b_2 -  \sqrt{\alpha(1-\alpha)} b_3}_{=: D_2(\alpha)}}.   \label{eq:lambmin_tempbd_1}  
\end{align}
Now it remains to lower bound the RHS of \eqref{eq:lambmin_tempbd_1}.

To this end, observe that $D_1(\alpha) \geq 0$, while 
\begin{equation} \label{eq:d2_non_neg_cond}
D_2(\alpha) \geq 0 \iff \alpha \geq \frac{b_3^2}{b_2^2 + b_3^2} =: \bar{\alpha}.
\end{equation}
Therefore we obtain
\begin{equation*}
    \max\set{{D_1(\alpha), D_2(\alpha)}} \geq 
     \begin{cases}
      D_1(\alpha) & \text{if $\alpha \leq \bar{\alpha}$,}\\
      \frac{D_1(\alpha) + D_2(\alpha)}{2} & \text{if $\alpha > \bar{\alpha}$.} 
    \end{cases}  
\end{equation*}
It is useful to observe that $\bar{\alpha} > 1/2$ since $b_2 < b_3$. Indeed, 
\begin{align*}
\lambmax(O_T^\top O_T) &= \max_{\norm{w}_2 = 1} w^\top \Big(\sum_{i=1}^T C_i^\top C_i \Big) w \leq \norm{C}_2^2 T \\
\implies \ b_2 &\leq \frac{\lambmax(O_T^\top O_T)}{T} \leq \Big(\frac{\lambmax{O_T^\top O_T}}{T} \Big)^{1/2} \norm{C}_2 < b_3.
\end{align*}

Through standard calculus, we can show (see Appendix \ref{appsec:low_bd_d1d2}) that for $\alpha > \bar{\alpha}$
\begin{equation}\label{eq:low_bd_d1d2avg}
    \frac{D_1(\alpha) + D_2(\alpha)}{2}
    \geq 
     \begin{cases}
      \frac{(1-\bar{\alpha})}{2}\mu b_1 & \text{if $\mu b_1 < b_2 + \frac{b_3^2 - b_2^2}{2b_2}$,}\\
      \frac{D_1(\bar{\alpha}_1) + D_2(\bar{\alpha}_1)}{2} & \text{if $\mu b_1 \geq b_2 + \frac{b_3^2 - b_2^2}{2b_2}$,} 
    \end{cases}  
\end{equation}
where 
\begin{align*}
    \bar{\alpha}_1 := \frac{1}{2} + \frac{1}{2}\sqrt{\frac{1}{1+4E^2}} \ \text{ with } \ E = \frac{b_3}{2(\mu b_1 - b_2)}. 
\end{align*}
Putting everything together, we have shown thus far that for any $\alpha \in [0,1]$,
\begin{equation*}
    \max\set{{D_1(\alpha), D_2(\alpha)}} \geq 
     \begin{cases}
 \frac{(1-\bar{\alpha})}{2}\mu b_1 & \text{if $\mu b_1 < b_2 + \frac{b_3^2 - b_2^2}{2b_2}$,} \\
\frac{(1-\bar{\alpha}_1)}{2}\mu b_1 + \underbrace{\frac{\bar{\alpha}_1}{2} b_2 -  \sqrt{\bar{\alpha}_1(1-\bar{\alpha}_1)} \frac{b_3}{2}}_{\geq 0 \text{ due to \eqref{eq:d2_non_neg_cond} and \eqref{eq:func_shape_temp_bd_1}}} & \text{if $\mu b_1 \geq b_2 + \frac{b_3^2 - b_2^2}{2b_2}$.}
\end{cases}  
\end{equation*}
Using the above bound in \eqref{eq:lambmin_tempbd_1}, we arrive at Lemma \ref{lem:low_bd_smallest_eig} after minor simplfications.
\end{proof}
%
%
%
\subsection{Upper bounding the estimation error} \label{subsec:est_err_bd}
We now complete the analysis by showing how Lemma \ref{lem:low_bd_smallest_eig} can be used to obtain the statement of Theorem \ref{thm:main_pen_ls}. 
Let $\eta \in \matR^{Tn}$ be formed by column-stacking $\eta_1,\dots,\eta_{T}$. Then we can write $y$ as
\begin{equation*}
    y =  C x + \eta.
\end{equation*}
Since $\mu M^\top M + C^\top C \succ 0$ under the conditions of the theorem, we obtain from \eqref{eq:sol_lin_sys} that
\begin{align*}  
    \est{x} &=  \Big(\mu M^\top M + C^\top C \Big)^{-1} C^\top y\\
    &= \Big(\mu M^\top M + C^\top C \Big)^{-1} (C^\top C)x + \Big(\mu M^\top M + C^\top C \Big)^{-1} C^\top \eta.
\end{align*}
This implies that $\est{x} - x$ can be bounded as follows.
\begin{align}
    \est{x} - x &= -\mu \Big(\mu M^\top M + C^\top C \Big)^{-1} (M^\top M)x \nonumber\\
    &+ \Big(\mu M^\top M + C^\top C \Big)^{-1} C^\top \eta \nonumber \\
    \implies \norm{\est{x}-x}_2^2 &\leq \underbrace{2\mu^2\norm{\Big(\mu M^\top M + C^\top C \Big)^{-1} (M^\top M)x}_2^2}_{=:E_1} \nonumber\\
    &+ \underbrace{2\norm{\Big(\mu M^\top M + C^\top C \Big)^{-1} C^\top \eta}_2^2}_{=:E_2}. \label{eq:est_err_main_thm_e1e2_decomp}
\end{align}
It remains to bound $E_1$ (bias error) and $E_2$ (variance error) separately. This is outlined in the lemmas below.
%
%
%
\begin{lemma}[Bias error] \label{lem:bias_err}
   For any $\mu > 0$, it holds that $E_1 \leq \frac{4\mu}{\lambar(\mu)} S_T$.
\end{lemma}

%
\begin{lemma}[Variance error] \label{lem:var_error}
    If $\eta$ is centered and $\sigma$-subgaussian, then for any $\mu > 0$, it holds with probability at least $1-
    \delta$, for $\delta \in (0,e^{-1})$, that
    \begin{align*}
        E_2 \leq 8n \sigma^2 \norm{C}_2^2 \left( \sum_{t=1}^{T-1} \frac{1}{(\lambar(\mu) + \mu\lambda_t)^2}  +  \frac{1}{\lambar^2(\mu)}   \right)\left(1+4\log\Big(\frac{1}{\delta} \Big)\right).
    \end{align*}
\end{lemma}
The proofs are outlined in Section \ref{sec:bias_var_err}. Using Lemma's \ref{lem:bias_err} and \ref{lem:var_error} in \eqref{eq:est_err_main_thm_e1e2_decomp} completes the proof.

\subsection{Bounding the bias and variance terms} \label{sec:bias_var_err}
Before proceeding, let us first note some useful consequences of Lemma \ref{lem:low_bd_smallest_eig} that will be essential in our analysis. Denote $(v_t)_{t=1}^{T}$ to be the eigenvectors of $L$ associated with the eigenvalues $(\lambda_t)_{t=1}^{T}$, where we recall that $\lambda_{T-1} > 0$ (since $G$ is connected) and $\lambda_T = 0$.
\begin{enumerate}
    \item Firstly, we obtain 
    \begin{align} \label{eq:low_bd_1}
        \mu M^\top M + C^\top C = \mu(L \otimes I_n) + C^\top C = \mu\sum_{t=1}^{T-1} \lambda_t [(v_t v_t^\top) \otimes I_n] + C^\top C \succeq \lambar(\mu) I_{Tn}
    \end{align}
 where the final L\"owner bound follows from Lemma \ref{lem:low_bd_smallest_eig}.
 
    \item Alternatively, we also have the L\"owner bound
    \begin{equation} \label{eq:low_bd_2}
        \mu M^\top M + C^\top C \succeq \mu M^\top M =  \sum_{t=1}^{T-1} (\mu \lambda_t) [(v_t v_t^\top) \otimes I_n].
    \end{equation}
    \end{enumerate}
   Hence using \eqref{eq:low_bd_1} and \eqref{eq:low_bd_2}, we obtain 
   \begin{equation}\label{eq:low_bd_fin}
       \mu M^\top M + C^\top C \succeq \sum_{t=1}^{T-1} \frac{(\lambar(\mu) + \mu\lambda_t)}{2} (v_t v_t^\top)\otimes I_n + \frac{\lambar(\mu)}{2} (v_T v_T^\top) \otimes I_n.
   \end{equation}
   Moreover, \eqref{eq:low_bd_fin} implies the following bounds on the eigenvalues of $\mu M^\top M + C^\top C$.
   \begin{align} \label{eq:all_eigval_bds}
       \lambda_i(\mu M^\top M + C^\top C)  \geq \frac{(\lambar(\mu) + \mu\lambda_t)}{2} \quad \text{for $(t-1)n + 1 \leq i \leq nt$}, \ t=1,\dots,T.  
   \end{align}
Finally, recall the standard fact that for positive definite matrices $P$ and $Q$, $P \succeq Q$ implies $P^{-1} \preceq Q^{-1}$. Using this together with \eqref{eq:low_bd_fin} implies 
\begin{align} \label{eq:inv_lowner_bd_fin}
           \Big(\mu M^\top M + C^\top C \Big)^{-1} \preceq \sum_{t=1}^{T-1} \frac{2}{(\lambar(\mu) + \mu\lambda_t)} [(v_t v_t^\top)\otimes I_n] + \frac{2}{\lambar(\mu)}  [(v_T v_T^\top)\otimes I_n].
\end{align}

\subsubsection{Proof of Lemma \ref{lem:bias_err}}
$E_1$ is bounded as follows.
\begin{align*}
    E_1 
    &= 2\mu^2 \norm{(\mu M^\top M + C^\top C)^{-1} (M^\top M) x}_2^2 \\
    &= 2\mu^2 \norm{(\mu M^\top M + C^\top C)^{-1/2} }_2^2 \norm{(\mu M^\top M + C^\top C)^{-1/2} (M^\top M) x}_2^2 \\
    &\leq \frac{2\mu^2}{\lambar(\mu)} \norm{(\mu M^\top M + C^\top C)^{-1/2} (M^\top M) x}_2^2 \tag{using \eqref{eq:low_bd_1}} \\
    &= \frac{2\mu^2}{\lambar(\mu)} x^\top (M^\top M) (\mu M^\top M + C^\top C)^{-1} (M^\top M) x \\
    &\leq \frac{2\mu^2}{\lambar(\mu)} \left( \sum_{t=1}^{T-1} \sum_{i=1}^{n}\Bigg( \frac{2}{(\lambar(\mu) + \mu\lambda_t)} \underbrace{\dotprod{v_t\otimes e_i}{(M^\top M) x}^2}_{= \lambda_t^2 \dotprod{v_t\otimes e_i}{x}^2} \Bigg)
    + \frac{2}{\lambar(\mu)}\sum_{i=1}^{n} \underbrace{\dotprod{v_T \otimes e_i}{(M^\top M) x}^2}_{=0} \right) \tag{using \eqref{eq:inv_lowner_bd_fin}} \\
    &= \frac{4\mu^2}{\lambar(\mu)} \sum_{t=1}^{T-1} \frac{\lambda_t^2 x^\top ((v_t v_t^\top) \otimes I_n) x}{\mu\lambda_t + \lambar(\mu)} \\
    &\leq \frac{4\mu}{\lambar(\mu)} \sum_{t=1}^{T-1} \lambda_t  x^\top ((v_t v_t^\top) \otimes I_n) x \\
    &\leq \frac{4\mu}{\lambar(\mu)} S_T.
\end{align*}

\subsubsection{Proof of Lemma \ref{lem:var_error}}
We begin by writing $E_2$ as $E_2 = 2\eta^\top \Sigma \eta$ where 
$$\Sigma :=  C \Big(\mu M^\top M + C^\top C \Big)^{-2} C^\top.$$

Invoking a concentration bound for random positive semidefinite quadratic forms with  centered subgaussian random vectors \cite[Thm. 2.1]{hsu2012tail}, it follows for any $t > 0$ that
\begin{align} \label{eq:rand_quad_tail_bd}
    \prob\left(\eta^\top \Sigma \eta \geq \sigma^2\left(\Tr(\Sigma) + 2\sqrt{\Tr(\Sigma^2) t} + 2\norm{\Sigma}_2 t \right)\right) \leq e^{-t}.
\end{align}
Now note that 
\begin{itemize}
    \item $\norm{\Sigma}_2 \leq \Tr(\Sigma)$ and,

    \item $\Tr(\Sigma^2) \leq \norm{\Sigma}_2\Tr(\Sigma) \leq (\Tr(\Sigma))^2$.
\end{itemize}
Plugging these in \eqref{eq:rand_quad_tail_bd}, we obtain the simplified bound 
\begin{align*}
        \prob\left(\eta^\top \Sigma \eta \leq \sigma^2 \Tr(\Sigma) (1+4t)\right)   \geq 1-e^{-t}, \qquad \text{if $t \geq 1$}.
\end{align*}
It remains to bound $\Tr(\Sigma)$, this is done below.
\begin{align*}
    \Tr(\Sigma) &= \Tr \Big(C \Big(\mu M^\top M + C^\top C \Big)^{-2} C^\top \Big) \\
&= \Tr\left((C^\top C)\Big(\mu M^\top M + C^\top C \Big)^{-2}\right) \\
&\leq \norm{C^\top C}_2 \Tr\Big(\big(\mu M^\top M + C^\top C \big)^{-2}\Big) \\
&\leq 4 \norm{C}_2^2 \left(\sum_{t=1}^{T-1} \frac{n}{(\lambar(\mu) + \mu\lambda_t)^2}  +  \frac{n}{\lambar^2(\mu)} \right) \tag{using \eqref{eq:all_eigval_bds}}
\end{align*}

%% file: exps.tex
%
\section{Numerical simulations} \label{sec:numerics}
We now perform simulations that demonstrate the weak-consistency of the penalized least-squares estimator, as $T \rightarrow \infty$, for complete and star graphs. 

The general setup is as follows. We fix $n$, $\sigma$, $S_T$ (as a function of $T$) and the graph $G$ (either complete or star graph). The aim is to plot the MSE for different values of $T$, where for each value of $T$, we average the MSE over $50$ Monte-Carlo runs. For each Monte-Carlo run, we perform the following steps.
\begin{enumerate}
    \item Generate $(C_t)_{t=1}^T$ -- either as in Proposition \ref{prop:rand_samp_model} (fixed $\theta$) or as in Proposition \ref{prop:erdos_renyi_meas_graph} (fixed $(p_t)_{t=1}^T$).

    \item Generate the ground-truth $x$ randomly, such it satisfies the smoothness condition as specified by $S_T$ (this is explained in more detail below). In the multi-layer translation synchronization setup, we subsequently  center\footnote{The centered $x_t$'s satisfy the smoothness condition with the same bound $S_T$ -- this is easy to verify.} each $x_t$.

    \item Obtain $y_t$ as per \eqref{eq:partial_meas_model} using $C_t$ and $x_t$ as generated above, with $\eta_t \stackrel{\text{iid}}{\sim} \calN(0,\sigma^2)$.

    \item Obtain $\est{x}$, either as a solution of \eqref{eq:sol_lin_sys}, or as in  \eqref{eq:cen_pen_estimator_trans_sync}, depending on the problem setting. Compute the MSE.
\end{enumerate}
To ensure that $x$ satisfies the smoothness condition \eqref{eq:smooth_cond_nodesigs}, we now describe a stochastic model which ensures that $x$ satisfies \eqref{eq:smooth_cond_nodesigs} in expectation.
\paragraph{Generating smooth signals for star graph.} We first sample $x_1 \sim \calN(0,\frac{1}{n} I_n)$ (where $\calN$ denotes the Gaussian distribution), which is the signal at the central node. Choosing $$\mu = \frac{1}{\sqrt{n}}(1,\cdots,1)^{\top} \in \matR^n, \quad \alpha = \sqrt{S_T/T} \ \text{and} \ m = \min\set{\lfloor S_T \rfloor, T}$$ the signals at the other nodes are then generated as
 \begin{equation*}
     x_t \stackrel{\text{i.i.d}}{\sim} \begin{cases}
    \calN(x_1,\frac{\alpha^2}{n} I_n) & \text{for $2 \leq t \leq T-m$}, \\
     \calN(x_1 + \mu,\frac{1}{n} I_n) & \text{for $T -m+1\leq t \leq T$}.
    \end{cases} 
 \end{equation*}
 Then it is easy to verify
 \begin{equation*}
     \sum_{t=2}^T \expec\big[\norm{x_t - x_1}_2^2\big] = \sum_{t=2}^{T-m} \alpha^2 + \sum_{t=T-m+1}^T \hspace{-1em}2 \ \lesssim S_T.
 \end{equation*}

 \paragraph{Generating smooth signals for complete graph.} We first generate $z \in \matR^n$ as
 \begin{equation*}
     z = \frac{\sqrt{S_T}}{T\sqrt{n}} (1,\cdots,1)^{\top}.
 \end{equation*}
Then the node-signals are obtained as
\begin{align*}
    x_t \stackrel{\text{i.i.d}}{\sim} \begin{cases}
    \calN(0,\frac{S_T}{T^2 n} I_n) & \text{for $1 \leq t \leq \lfloor T/2 \rfloor$}, \\
     \calN(z, \frac{S_T}{T^2 n} I_n) & \text{for $\lfloor T/2 \rfloor+1 \leq t \leq T$}.
    \end{cases} 
\end{align*}
It is again easily verified that 
 \begin{equation*}
     \sum_{t < t'} \expec\big[\norm{x_t - x_{t'}}_2^2\big] \lesssim S_T + T^2\norm{z}_2^2 \lesssim S_T.
 \end{equation*}
\paragraph{Results.} We obtain the following results.
\begin{enumerate}
  \item The results in Figure \ref{fig:mse_versus_T_star_complete} 
  are for the setting where $C_t$ is generated as in Proposition \ref{prop:rand_samp_model} and $\est{x}$ is obtained via \eqref{eq:sol_lin_sys}. This is shown for different smoothness regimes, with $n = 5$, $\sigma = 1$ and  $\theta \in \set{0.2, 0.5}$. Here $\mu$ is chosen as in Corollary \ref{corr:rand_samp_graphs} with $c_1 = 2$ for the complete graph, and $c_1 = 3$ for the star graph. As expected, the MSE goes to zero as $T$ increases. For large values of $S_T$, we expect the MSE to be relatively larger -- this is seen more clearly in the case of a star graph, while the error values (across different $S_T$) are quite similar for a complete graph.  
  
  \item The results in Figure \ref{fig:mse_versus_T_star_complete_trans_sync} are for the multi-layer translation synchronization model in Section \ref{sec:disc_multlayer_transsync}. For simplicity, we consider $G_t$ to be i.i.d Erd\"os-Renyi graphs with $p_t = p$ for all $t$. Moreover, we choose $n = 50$ and $\sigma = 1$ throughout. The plots are shown for $p = 1/n = 0.02$, and for $p = 1/(5n) = 0.004$ -- both corresponding to regimes in which some fraction of the individual graphs will be disconnected with high probability. In both cases the MSE goes to zero as $T$ increases, with the error being higher for smaller $p$ as expected.
\end{enumerate}

%
%
\begin{figure}[!htp]
\centering
\begin{subfigure}{.5\textwidth}
  \centering
  \includegraphics[width=0.9\linewidth]{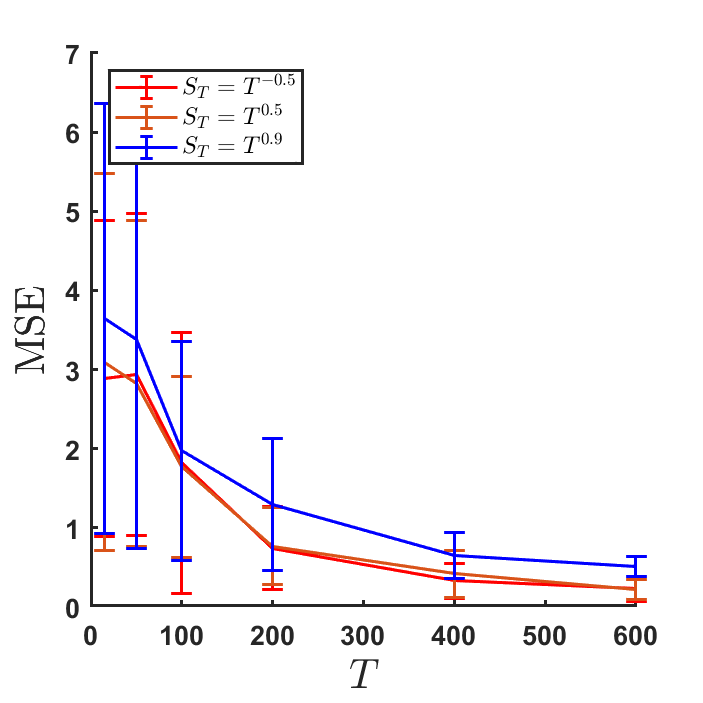}
  \caption{Star graph, $\theta = 0.2$}
\end{subfigure}%
\begin{subfigure}{.5\textwidth}
  \centering
  \includegraphics[width=0.9\linewidth]{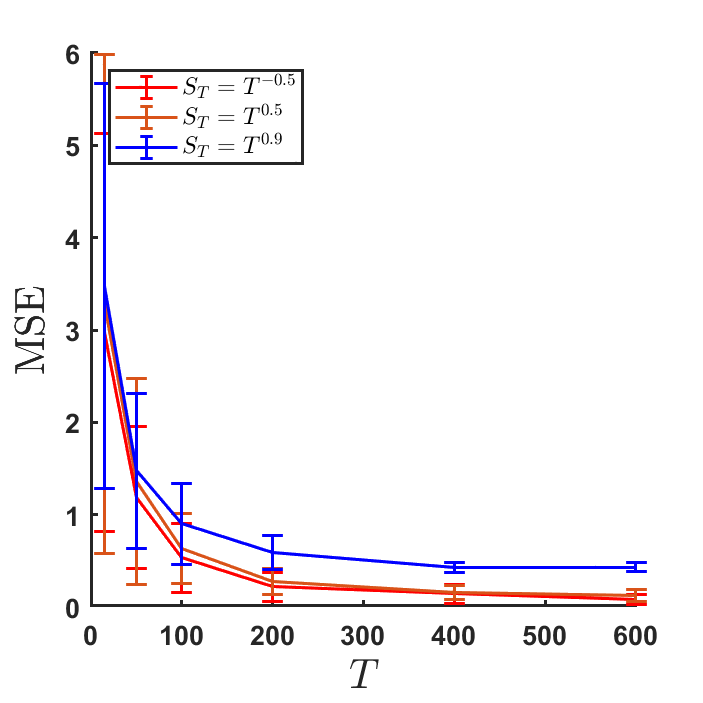}
  \caption{Star graph, $\theta = 0.5$}
\end{subfigure}%
\hfill
\begin{subfigure}{.5\textwidth}
  \centering
  \includegraphics[width=0.9\linewidth]{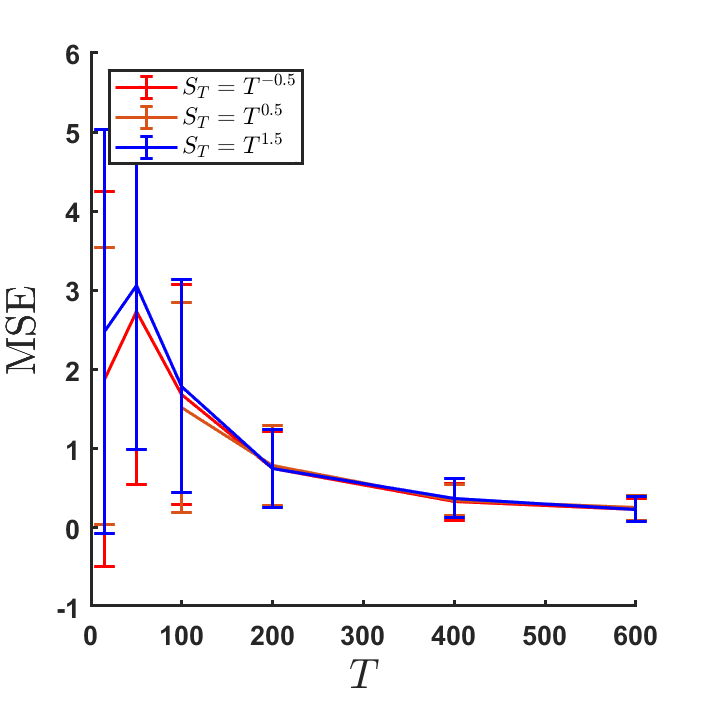}
  \caption{Complete graph, $\theta = 0.2$}
\end{subfigure}%
\begin{subfigure}{.5\textwidth}
  \centering
  \includegraphics[width=0.9\linewidth]{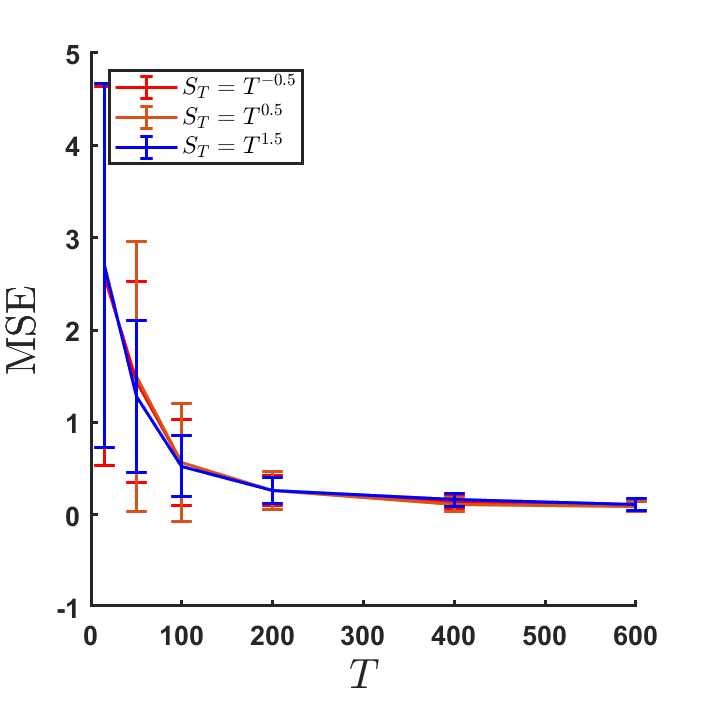}
  \caption{Complete graph, $\theta = 0.5$}
\end{subfigure}
\caption{MSE versus $T$ for star graph (top row) and complete graph (bottom row), for fixed $\theta \in \set{0.2, 0.5}$ with $C_t$ generated as in Proposition \ref{prop:rand_samp_model}.      Here, $\est{x}$ is obtained as a solution of \eqref{eq:sol_lin_sys}. We set $n=5$ and $\sigma = 1$. The MSE is averaged over $50$ Monte Carlo trials.}
\label{fig:mse_versus_T_star_complete}
\end{figure}

%
%
%
\begin{figure}[!htp]
\centering
\begin{subfigure}{.5\textwidth}
  \centering
  \includegraphics[width=0.9\linewidth]{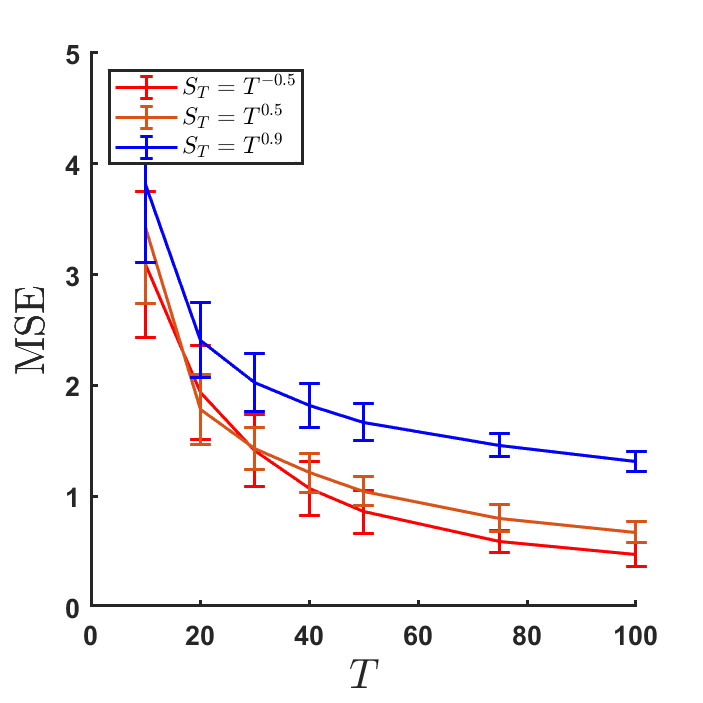}
  \caption{Star graph, $p = 0.02$}
\end{subfigure}%
\begin{subfigure}{.5\textwidth}
  \centering
  \includegraphics[width=0.9\linewidth]{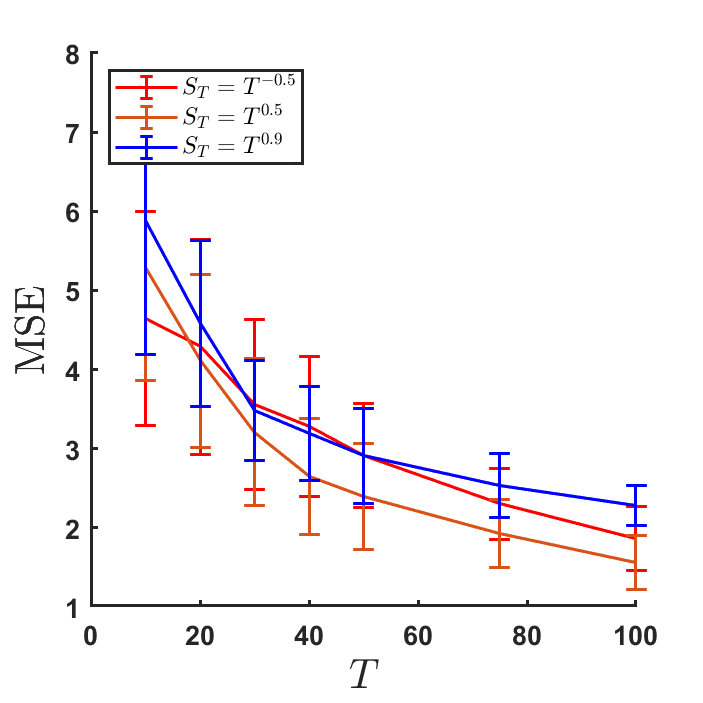}
  \caption{Star graph, $p = 0.004$}
\end{subfigure}%
\hfill
\begin{subfigure}{.5\textwidth}
  \centering
  \includegraphics[width=0.9\linewidth]{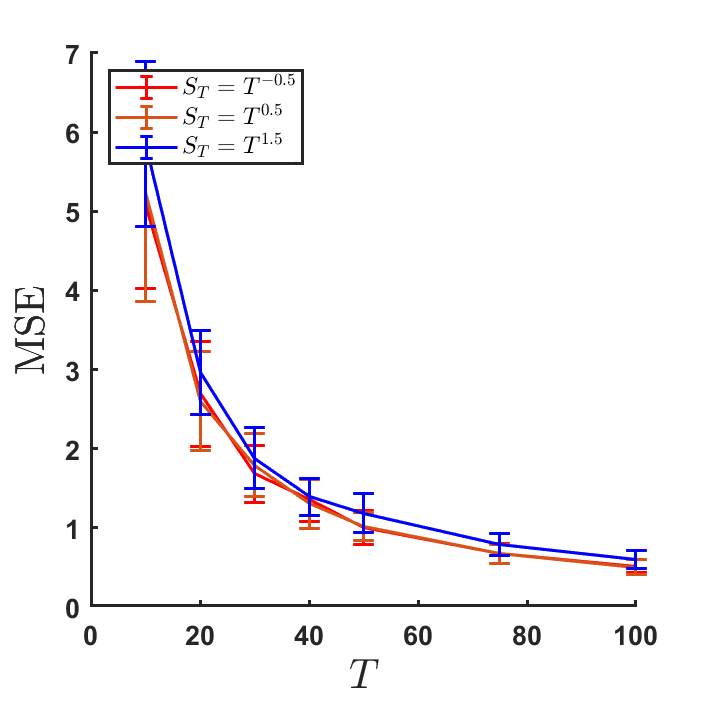}
  \caption{Complete graph, $p = 0.02$}
\end{subfigure}%
\begin{subfigure}{.5\textwidth}
  \centering
  \includegraphics[width=0.9\linewidth]{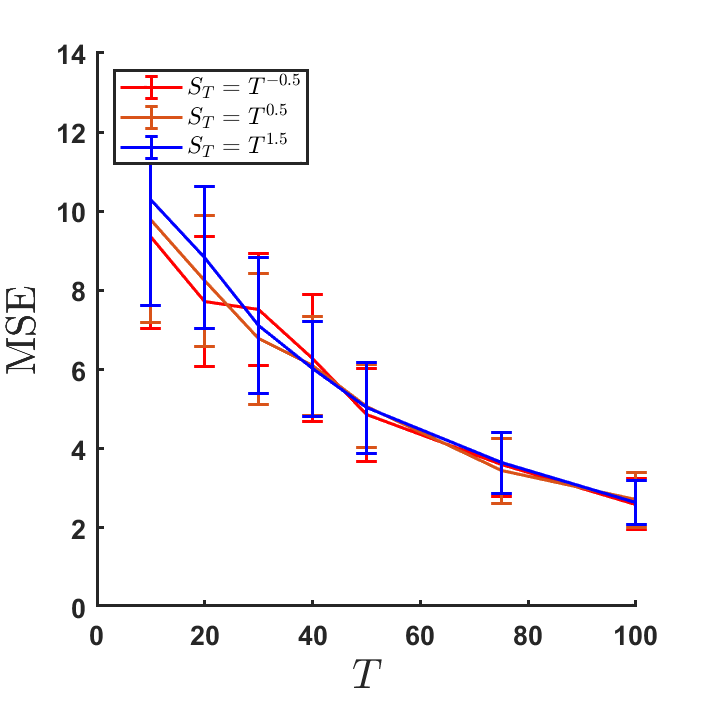}
  \caption{Complete graph, $p = 0.004$}
\end{subfigure}
\caption{MSE versus $T$ for star graph (top row) and complete graph (bottom row), for the multi-layer translation synchronization model. Here, $G_t$ are independent Erd\"os-Renyi graphs with $p_t = p$ for all $t$, where $p \in \set{0.02, 0.004}$. Also, $n=50$ and $\sigma = 1$. The estimate $\est{x}$ is found using \eqref{eq:cen_pen_estimator_trans_sync}. The MSE is averaged over $50$ Monte Carlo trials.}
\label{fig:mse_versus_T_star_complete_trans_sync}
\end{figure}

%% file: appendix_analysis.tex
%
\section{Proof of Proposition \ref{prop:uniq_sol_pen_est}} \label{app:proof_prop_uniq}
Recall that for any pair of positive semidefinite matrices $P, Q$, we have $\nullsp(P+Q) = \nullsp(P) \cap \nullsp(Q)$. Hence if $\mu > 0$, then this implies 
\begin{equation} \label{eq:nullsp_sum_psd}
\nullsp(\mu M^\top M + C^\top C) = \nullsp(M^\top M) \cap \nullsp(C^\top C).
\end{equation}
Now the null space of $M^\top M$ can be written as
\begin{equation} \label{eq:null_sp_mm}
\nullsp(M^\top M) = \nullsp(M) = \set{
\begin{pmatrix}
I \\
I  \\
\vdots \\
I
\end{pmatrix} z: z\in \matR^n}    
\end{equation}
and has dimension $n$. If $\text{rank}(O_T) = n$, then clearly no non-zero vector in $\nullsp(M^\top M)$ can lie in $\nullsp(C^\top C)$ (= $\nullsp(C)$). 
%

%
\section{Proof of Proposition \ref{prop:rand_samp_model}} \label{appsec:proof_rand_samp_model}
Writing $O_T^\top O_T = \sum_{t=1}^T C_t^\top C_t$ and denoting $Y_t := C_t^\top C_t$ we apply the Matrix Chernoff bound of in \citep[Theorem 5.1.1]{Tropp15MatConc}. 

Since $\expec[\sum_{t=1}^T Y_t] = \frac{\theta}{n} I_n$, we first obtain for any $u \in (0,1)$ that 
\begin{align*}
    \prob\Big(\lambmin(O_T^\top O_T) \leq u\frac{\theta T}{n}\Big) \leq n\exp\Big(-\frac{(1-u)^2\theta T}{2n}\Big)
\end{align*}
Taking $u = 1/2$, we thus have that $\lambmin(O_T^\top O_T) \geq \frac{\theta T}{2n}$ with probability at least $1-\delta$ if $T \geq \frac{8n}{\theta}\log(n/\delta)$. Next, we have the following inequality for $\lambmaxup(O_T^\top O_T)$, for $v \geq e$.
\begin{align*}
\prob\Big(\lambmaxup(O_T^\top O_T) \geq v\frac{\theta T}{n}\Big) \leq n \Big(\frac{e}{v} \Big)^{\frac{v \theta T}{n}}. 
\end{align*}
Choosing $v = 2e$, we thus obtain with probability at least $1-\delta$ that $\lambmaxup(O_T^\top O_T) \leq 2e\frac{\theta T}{n}$, provided 
\begin{align*}
    T \geq \frac{n}{2e\theta} \log_2(n/\delta) = \Big(\frac{1}{2e\log 2} \Big)\frac{n}{\theta}\log\Big(\frac{n}{\delta} \Big).
\end{align*}
The statement follows readily via a union bound.

\section{Lower bound in \eqref{eq:low_bd_d1d2avg}} \label{appsec:low_bd_d1d2}
Denote $g(\alpha) := \frac{D_1(\alpha) + D_2(\alpha)}{2}$ for convenience, where $\alpha \in [\bar{\alpha},1]$, and we recall that $\bar{\alpha} = \frac{b_3^2}{b_2^2 + b_3^2}$. We then have
\begin{align*}
    g(\alpha) 
    &= \left(\frac{1-\alpha}{2} \right) \mu b_1 + \frac{\alpha}{2} b_2 - \sqrt{\alpha(1-\alpha)} \frac{b_3}{2} \\
    &= \frac{\mu b_1}{2} + \left(\frac{b_2 - \mu b_1}{2} \right)\alpha - \sqrt{\alpha(1-\alpha)}\frac{b_3}{2}.
\end{align*}
Let us look at the shape of $g(\alpha)$ in $[\bar{\alpha}, 1]$. Note that $$g'(\alpha) = \frac{b_2 - \mu b_1}{2} - \frac{b_3(1-2\alpha)}{4\sqrt{\alpha(1-\alpha)}}.$$ 
There are two cases to consider.
\begin{enumerate}
    \item \underline{Case $1$:} If ${b_2 \geq \mu b_1}$ then $g'(\alpha) > 0$ for $\alpha \in [\bar{\alpha}, 1]$ since $\bar{\alpha} > 1/2$. Hence $g(\bar{\alpha})$ will be the minimum. 

    \item \underline{Case $2$:} Suppose $b_2 < \mu b_1$. Now note that 
    \begin{align*}
        g'(\alpha) \geq 0 &\iff \frac{b_2 - \mu b_1}{2} \geq \frac{b_3(1-2\alpha)}{4\sqrt{\alpha(1-\alpha)}} \\
        &\iff \frac{\alpha(1-\alpha)}{2\alpha - 1} \leq \underbrace{\frac{b_3}{2(\mu b_1 - b_2)}}_{=: E} \\
        &\iff \alpha^2 - \alpha + \frac{E^2}{1 + 4E^2} \geq 0 \\
        &\iff \alpha \geq \frac{1}{2} + \frac{1}{2}\sqrt{\frac{1}{1+4E^2}} \ =: \bar{\alpha}_1
    \end{align*}
    where the last inequality holds since $\alpha \geq \bar{\alpha} > 1/2$. In particular, $g'(\alpha) = 0$ iff $\alpha = \bar{\alpha}_1$. Moreover, one can readily verify that 
    \begin{equation} \label{eq:func_shape_temp_bd_1}
        \bar{\alpha}_1 \geq \bar{\alpha} \iff \mu b_1 \geq b_2 + \frac{b_3^2 - b_2^2}{2 b_2} \quad (> b_2).
    \end{equation}
    Thus we can make two observations. Firstly, if \eqref{eq:func_shape_temp_bd_1} holds, then it implies $g'(\bar{\alpha}_1) = 0$ and $g'(\alpha) > 0$ for $\alpha \in (\bar{\alpha}_1, 1]$. Moreover, $g'(\alpha) < 0$ for $\alpha < \bar{\alpha}_1$. Consequently, $g(\bar{\alpha}_1)$ will be the minimum in $[\bar{\alpha}, 1]$. On the other hand, if $\bar{\alpha}_1 < \bar{\alpha}$, i.e.,
$$b_2 < \mu b_1 < b_2 + \frac{b_3^2 - b_2^2}{2 b_2},$$
then $g'(\alpha) > 0$ for $\alpha \in [\bar{\alpha}, 1]$ and $g(\bar{\alpha})$ is the minimum.
\end{enumerate}
To summarize, we have shown that if $\mu b_1 < b_2 + \frac{b_3^2 - b_2^2}{2 b_2}$, then $g(\bar{\alpha}) = (\frac{1-\bar{\alpha}}{2}) \mu b_1$ is the minimum of $g(\cdot)$. If $\mu b_1 \geq b_2 + \frac{b_3^2 - b_2^2}{2 b_2}$ then $g(\bar{\alpha}_1)$ is the minimum of $g(\cdot)$. Hence we obtain the stated lower bound on $g(\cdot)$.

%
%

\section{Corollaries of Theorem \ref{thm:main_pen_ls}} \label{appsec:proofs_corrs_spec_graphs_detsamp}
%
\subsection{Proof of Corollary \ref{corr:compl_graph}} \label{appsec:proof_corr_compl_graph}
For the complete graph $\lambda_t = T- 1$ for $1 \leq t \leq T-1$ which means $b_1 = T-1$. Hence from Theorem \ref{thm:main_pen_ls}, we have 
\begin{equation} \label{eq:mu_tmp1_compl}
    \mu \gtrsim \frac{1}{T}\left(\frac{\lambminlow(O_T^\top O_T)}{T} + \norm{C}_2^2\frac{\lambmaxup(O_T^\top O_T)}{\lambminlow(O_T^\top O_T)} \right) \implies \lambarp(\mu) \gtrsim \frac{\lambminlow(O_T^\top O_T)}{ T}.
\end{equation}
Using $\lambar(\mu) \geq \lambarp(\mu)$, Theorem \ref{thm:main_pen_ls} then yields the estimation error bound
\begin{align} 
    &\norm{\est{x} - x}_2^2 \nonumber \\
   &\leq \frac{4\mu}{\lambar(\mu)} S_T + 40 n\sigma^2 \norm{C}_2^2 \left(   \frac{T-1}{(\lambar(\mu) + \mu(T-1))^2}  +  \frac{1}{\lambar^2(\mu)}   \right) \log\Big(\frac{1}{\delta}\Big)  \nonumber \\
    &\lesssim \frac{\mu T S_T}{\lambminlow(O_T^\top O_T)} + n\sigma^2 \norm{C}_2^2\Bigg(\frac{T}{(\frac{\lambminlow(O_T^\top O_T)}{T} + \mu T)^2} +\frac{T^2}{\lambminlow^2(O_T^\top O_T)} \Bigg) \log\Big(\frac{1}{\delta}\Big)  \label{eq:tmp_esterr_1}
\end{align}
Denote $f$ to be
\begin{equation*}
    f(\mu) = \frac{\mu T S_T}{\lambminlow(O_T^\top O_T)}  + n\sigma^2 \norm{C}_2^2\Bigg(\frac{T}{\big(\frac{\lambminlow(O_T^\top O_T)}{T} + \mu T \big)^2}\Bigg).
\end{equation*}
Clearly $f$ is convex w.r.t $\mu$, and its unconstrained minimizer is readily verified to be 
\begin{equation*}
    \mu = (2n\sigma^2 \norm{C}_2^2)^{1/3}\Big(\frac{T}{S_T} \Big)^{1/3} \frac{\lambminlow^{1/3}(O_T^\top O_T)}{T} - \frac{\lambminlow(O_T^\top O_T)}{T^2}.
\end{equation*}
The above minimizer could be negative, but since we also require $\mu$ to satisfy \eqref{eq:mu_tmp1_compl}, we thus choose $\mu$ as in \eqref{eq:mu_compl_graph}. Using the bounds
\begin{align*}
    \frac{\lambminlow(O_T^\top O_T)}{T} + \mu^* T 
    &\geq  (2 n \sigma^2 \norm{C}_2^2)^{1/3} \Big(\frac{T}{S_T}\Big)^{1/3} \lambminlow^{1/3}(O_T^\top O_T) , \\
    \mu^* T &\leq (2 n \sigma^2 \norm{C}_2^2)^{1/3} \Big(\frac{T}{S_T}\Big)^{1/3} \lambminlow^{1/3}(O_T^\top O_T) + c_1\Big(\norm{C}_2^2\frac{\lambmaxup(O_T^\top O_T)}{\lambminlow(O_T^\top O_T)} + \frac{\lambminlow(O_T^\top O_T)}{T} \Big)
\end{align*}
in \eqref{eq:tmp_esterr_1} the statement of Corollary \ref{corr:compl_graph} follows after minor simplifications. 

\subsection{Proof of Corollary \ref{corr:star_graph}} \label{appsec:proof_corr_star_graph}
For the star graph $\lambda_1 = T-1$ and $\lambda_t =  1$ for $2 \leq t \leq T-1$, which means $b_1 = 1$. Hence from Theorem \ref{thm:main_pen_ls}, we have 
\begin{equation} \label{eq:mu_tmp1_star}
    \mu \gtrsim \left(\frac{\lambminlow(O_T^\top O_T)}{T} + \norm{C}_2^2\frac{\lambmaxup(O_T^\top O_T)}{\lambminlow(O_T^\top O_T)} \right) \implies \lambarp(\mu) \gtrsim \frac{\lambminlow(O_T^\top O_T)}{ T}.
\end{equation}
Using $\lambar(\mu) \geq \lambarp(\mu)$, Theorem \ref{thm:main_pen_ls} then yields the estimation error bound
\begin{align} 
    &\norm{\est{x} - x}_2^2 \nonumber \\
   &\leq \frac{4\mu}{\lambar(\mu)} S_T + 40 n\sigma^2 \norm{C}_2^2 \left(   \frac{T-2}{(\lambar(\mu) + \mu)^2}  + \frac{1}{(\lambar(\mu) + \mu T)^2} + \frac{1}{\lambar^2(\mu)}   \right)  \log\Big(\frac{1}{\delta}\Big)  \nonumber \\
    &\lesssim \frac{\mu T S_T}{\lambminlow(O_T^\top O_T)} + n\sigma^2 \norm{C}_2^2\Bigg(\frac{T}{\Big(\frac{\lambminlow(O_T^\top O_T)}{T} + \mu \Big)^2} + \frac{1}{\Big(\frac{\lambminlow(O_T^\top O_T)}{T} + \mu T \Big)^2} \nonumber \\
    &+ \frac{T^2}{\lambminlow^2(O_T^\top O_T)} \Bigg) \log\Big(\frac{1}{\delta} \Big) \nonumber  \\
    &\lesssim  \frac{\mu T S_T}{\lambminlow(O_T^\top O_T)} + n\sigma^2 \norm{C}_2^2\Bigg(\frac{T}{\Big(\frac{\lambminlow(O_T^\top O_T)}{T} + \mu \Big)^2} +  \frac{T^2}{\lambminlow^2(O_T^\top O_T)} \Bigg) \log\Big(\frac{1}{\delta}\Big) 
    \label{eq:tmp_esterr_2}
\end{align}
where the final inequality follows since $\frac{\lambminlow(O_T^\top O_T)}{T} + \mu \leq \frac{\lambminlow(O_T^\top O_T)}{\sqrt{T}} + \mu T^{3/2}$.

Proceeding analogously to the proof of Corollary \ref{corr:compl_graph}, we denote $f$ to be
\begin{equation*}
    f(\mu) = \frac{\mu T S_T}{\lambminlow(O_T^\top O_T)} + n\sigma^2 \norm{C}_2^2 \frac{T}{\Big(\frac{\lambminlow(O_T^\top O_T)}{T} + \mu \Big)^2}
\end{equation*}
and find its global minimizer. This along with the condition on $\mu$ in \eqref{eq:mu_tmp1_star} leads to the choice of $\mu = \mu^*$ in \eqref{eq:mu_star_graph}. Then using the bounds
\begin{align*}
    \frac{\lambminlow(O_T^\top O_T)}{T} + \mu^*  
    &\geq  \frac{(2 n \sigma^2 \norm{C}_2^2)^{1/3}  \lambminlow^{1/3}(O_T^\top O_T)}{S_T^{1/3}}, \\
    \mu^* T &\leq \frac{T (2 n \sigma^2 \norm{C}_2^2)^{1/3}  \lambminlow^{1/3}(O_T^\top O_T)}{S_T^{1/3}} + c_1\Big(T \norm{C}_2^2\frac{\lambmaxup(O_T^\top O_T)}{\lambminlow(O_T^\top O_T)} + \lambminlow(O_T^\top O_T) \Big)
\end{align*}
in \eqref{eq:tmp_esterr_2} leads to the stated error bound after minor simplifications.

%
\section{Proofs from Section \ref{sec:disc_multlayer_transsync}} \label{appsec:proofs_trans_sync}
\subsection{Proof of Proposition \ref{prop:uniq_sol_est_trans_sync}}
Starting from \eqref{eq:nullsp_sum_psd}, and using the fact $\nullsp(O_T) = \spansp(\ones_{n})$, we simply note that for any $u = [z^\top \cdots z^{\top}]^\top \in \nullsp(M)$ (with $z \in \matR^n$), we have $C u = O_T z = 0$ iff $z \in \spansp(\ones_n)$.

%
\subsection{Proof of Theorem \ref{thm:main_pen_ls_trans_sync}}
The proof steps are along the same lines as that of Theorem \ref{thm:main_pen_ls}, we highlight here the main differences. 
\subsubsection{Lower bounding $\lambda_{T(n-1)}\Big(P(\mu M^\top M + C^\top C)P \Big)$}
Since $\spansp(P) \subset \spansp(\mu M^{\top}M + C^\top C)$, we have
$$\lambda_{T(n-1)}\Big(P(\mu M^\top M + C^\top C)P \Big) \geq \lambda_{nT-1}(\mu M^\top M + C^\top C),$$
hence we focus on lower bounding $\lambda_{nT-1}(\mu M^\top M + C^\top C)$ from now. Start by writing
\begin{equation} \label{eq:var_form_secsmall_eig}
    \lambda_{nT-1}(\mu M^\top M +   C^\top C) = \min_{\stackrel{\norm{u}_2^2 = 1}{u \perp \spansp(\ones_{nT})}} u^\top (\mu M^\top M + C^\top C) u.
\end{equation}  
For any $u \perp \spansp(\ones_{nT})$ we can decompose $u$ as $u = u_1 + u_2$ where 
\begin{equation*}
    u_1 \in \spansp^{\perp}(\ones_{nT}) \cap \nullsp(M^\top M) \quad \text{ and } \quad  u_2 \in \nullsp^{\perp}(M^\top M)
\end{equation*}
since $\spansp(\ones_{nT}) \subset \nullsp(M^\top M)$. Notice that any such $u_1$ will be of the form
\begin{equation*}
    u_1 = \begin{pmatrix}
z \\
z  \\
\vdots \\
z
\end{pmatrix} \quad \text{where} \quad \ones_n^\top z = 0.
\end{equation*}
Now consider for $\alpha \in [0,1]$ the set 
\begin{align*}
    \calS_{\alpha} := \set{u=u_1+u_2: \spansp^{\perp}(\ones_{nT}) \cap \nullsp(M^\top M), u_2 \in \nullsp^{\perp}(M^\top M), \norm{u_1}_2^2 = \alpha, \norm{u_2}_2^2 = 1-\alpha}.
\end{align*}
Then using \eqref{eq:var_form_secsmall_eig} and \eqref{eq:var_form_ineq}, we obtain (as before) the bound 
\begin{align} 
    &\lambda_{nT-1}(\mu M^\top M + C^\top C) \nonumber \\
    &= \min_{\alpha \in [0,1]} \min_{u \in \calS_{\alpha}} u^\top (\mu M^\top M + C^\top C) u \nonumber \\
    &\geq \min_{\alpha \in [0,1]} \min_{u \in \calS_{\alpha}} \max\set{\mu u_2^\top  M^\top M u_2,\  u_1^\top (C^\top C)u_1 - \abs{2u_1^\top (C^\top C) u_2}}. \label{eq:var_form_secsmall_eig_new}
\end{align}
Notice that any $u_1 \in \spansp(\ones_{nT}) \subset \nullsp(M^\top M)$ with $\norm{u_1}_2^2 = \alpha$ can be written as 
\begin{equation*}
    u_1 = \frac{\sqrt{\alpha}}{\sqrt{T}} \begin{pmatrix}
V \\
V  \\
\vdots \\
V
\end{pmatrix}\frac{w}{\norm{w}_2},
\end{equation*}
where $V \in \matR^{n \times(n-1)}$ is any orthonormal basis for $\spansp^{\perp}(\ones_{n})$. Then repeating the steps in the proof of Lemma \ref{lem:low_bd_smallest_eig}, with $b_2$ now defined as in \eqref{eq:b2_def_transync}, it is easy to verify that $\lambda_{nT-1}(\mu M^\top M + C^\top C) \geq \lambarp(\mu)$ holds for $\lambarp(\mu)$ as defined in Lemma \ref{lem:low_bd_smallest_eig}.

On the other hand, we claim that $\lambda_{T(n-1)}\big(P(\mu M^\top M + C^\top C)P \big) \geq \lambda_{(n-1)T}(C^\top C)$ also holds. Indeed, denote $$\beta := \min_{t\in [T]} \lambda_{n-1}(C_t^\top C_t) = \lambda_{(n-1)T}(C^\top C),$$ and let $(u_j)_{j=1}^{n-1}$ be any orthonormal basis for $\spansp^{\perp}(\ones_{n})$. We then have $C_t^\top C_t \succeq \beta \sum_{j=1}^{n-1} u_j u_j^\top$ for each $t$, which implies the L\"owner bound 
\begin{align*}
    C^\top C 
    \succeq \beta \left[I_T \otimes \Big(\sum_{j=1}^{n-1} u_j u_j^\top\Big) \right] 
    = \beta \Big(I_T \otimes (I_n - \frac{1}{n}\ones_n\ones_n^\top) \Big) = \beta P
\end{align*}
using the definition of the projection operator $P$. Using the fact $P(\mu M^\top M + C^\top C)P \succeq P (C^\top C) P$, we obtain $P(\mu M^\top M + C^\top C)P \succeq \beta P$, which proves the aforementioned claim. 

%
\subsubsection{Bounding the estimation error}
Since $\spansp(P(\mu M^\top M + C^\top C)P) = \spansp(P)$ and $x \in \spansp(P)$ (by assumption), this implies
\begin{align}
    \est{x} - x &= -\mu  \Big(P(\mu M^\top M + C^\top C) P \Big)^{\dagger} (M^\top M) x 
    + \Big(P(\mu M^\top M + C^\top C)P \Big)^{\dagger}  C^\top \eta \nonumber \\
    &\leq \underbrace{2\mu^2\norm{ \Big(P(\mu M^\top M + C^\top C) P \Big)^{\dagger} (M^\top M) x}_2^2}_{=:E_1} 
    + \underbrace{2\norm{\Big(P(\mu M^\top M + C^\top C)P \Big)^{\dagger}  C^\top \eta}_2^2}_{=:E_2}.    \label{eq:est_err_transync_e1e2_decomp}
\end{align}
As in Section \ref{sec:bias_var_err}, we obtain L\"owner bounds on $\mu M^\top M + C^\top C$ as follows. Recall that $(v_t)_{t=1}^T$ denote the eigenvectors of $L$ associated with the eigenvalues $\lambda_1 \geq \cdots \lambda_T$. Let $(u_j)_{j=1}^{n-1}$ be any orthonormal basis for $\spansp^{\perp}(\ones_{n})$. Since 
\begin{equation*}
    \spansp(\mu M^\top M + C^\top C) = \nullsp^\perp(M^\top M) \bigoplus \set{\nullsp(M^\top M) \cap \spansp^{\perp}(\ones_{nT})}
\end{equation*}
and $\lambda_{nT-1}(\mu M^\top M + C^\top C) \geq \lambar(\mu)$ we obtain
\begin{equation*}
    \mu M^\top M + C^\top C \succeq \lambar(\mu)\left[\sum_{t=1}^{T-1} (v_t v_t^\top) \otimes I_n + \sum_{j=1}^{n-1} (v_T v_T^\top) \otimes (u_j u_j^\top) \right],
\end{equation*}
which together with \eqref{eq:low_bd_2} yields
   \begin{equation*} 
       \mu M^\top M + C^\top C \succeq \sum_{t=1}^{T-1} \frac{(\lambar(\mu) + \mu\lambda_t)}{2} (v_t v_t^\top)\otimes I_n + \frac{\lambar(\mu)}{2} \sum_{j=1}^{n-1} (v_T v_T^\top) \otimes (u_j u_j^\top).
   \end{equation*}
   Since $P = I_T \otimes (I_n - \frac{1}{n}\ones_n\ones_n^\top)$ this implies
   \begin{align}\label{eq:low_bd_transyn_temp1}
       P(\mu M^\top M &+ C^\top C) P \nonumber \\
       &\succeq \sum_{t=1}^{T-1} \frac{(\lambar(\mu) + \mu\lambda_t)}{2} (v_t v_t^\top)\otimes (I_n - \frac{1}{n}\ones_n\ones_n^\top) + \frac{\lambar(\mu)}{2} \sum_{j=1}^{n-1} (v_T v_T^\top) \otimes (u_j u_j^\top).
   \end{align}
Moreover, \eqref{eq:low_bd_transyn_temp1} implies  
\begin{align} \label{eq:all_eigval_bds_transync}
       \lambda_i(P(\mu M^\top M + C^\top C) P) 
       \geq 
       \begin{cases}
      \frac{(\lambar(\mu) + \mu\lambda_t)}{2} & \text{for $(t-1)(n-1) + 1 \leq i \leq (n-1)t$}, \ t=1,\dots,T-1, \\
      \frac{\lambar(\mu)}{2} & \text{for $(T-1)(n-1) + 1 \leq i \leq (n-1)T$.}
       \end{cases}
\end{align}
Finally, since $P(\mu M^\top M + C^\top C)P$ and the RHS of \eqref{eq:low_bd_transyn_temp1} share the same null space, \eqref{eq:low_bd_transyn_temp1} implies\footnote{For matrices $A,B \succeq 0$ with $\nullsp(A) = \nullsp(B)$, we have that $A \succeq B$ implies $A^\dagger \preceq B^\dagger$.} 
\begin{align}
       \Big(P(\mu M^\top M &+ C^\top C)P \Big)^\dagger \nonumber \\
       &\preceq \sum_{t=1}^{T-1} \frac{2}{(\lambar(\mu) + \mu\lambda_t)} (v_t v_t^\top)\otimes (I_n - \frac{1}{n}\ones_n\ones_n^\top) + \frac{2}{\lambar(\mu)} \sum_{j=1}^{n-1} (v_T v_T^\top) \otimes (u_j u_j^\top). \label{eq:low_bd_transyn_temp2}
\end{align}
Using \eqref{eq:low_bd_transyn_temp1}, \eqref{eq:all_eigval_bds_transync} and \eqref{eq:low_bd_transyn_temp2}, we can then proceed analogously as in Section \ref{subsec:est_err_bd} to obtain the same bounds on $E_1, E_2$ as in Lemmas \ref{lem:bias_err} and \ref{lem:var_error}. Plugging these in \eqref{eq:est_err_transync_e1e2_decomp} leads to the statement of the theorem.

%
\subsection{Proof of Proposition \ref{prop:erdos_renyi_meas_graph}} \label{appsec:proof_ER_meas_graph}
%
\subsubsection{Proof of part 1}
Note that $L_t := C_t^\top C_t$ is the Laplacian matrix of $G_t$.  Let $D_t$ denote the diagonal matrix of the degrees of its vertices, and $A_t$ denote its adjacency matrix. Then $L_t = D_t - A_t$. Denote $\Lbar_t, \Dbar_t$ and $\Abar$ to be the expectations of these matrices. Then, 
\begin{align} \label{eq:L_conc_transync}
    \norm{\sum_t(L_t - \Lbar_t)}_2 \leq \norm{\sum_t(D_t - \Dbar_t)}_2  + \norm{\sum_t(A_t - \Abar_t)}_2 
\end{align}
and part $1$ will follow after suitably bounding the RHS terms above. 
\paragraph{Bounding $\norm{\sum_t(D_t - \Dbar_t)}_2 $.}
Since $(\Dbar)_{ii} = (n-1)\psum$, we have by an application of standard Chernoff bounds for sums of independent Bernoulli variables (e.g., \cite{mitzenmacher2017probability}) that for any $\delta' \in (0,1)$,
\begin{align*}
    \prob\left(\Big|\sum_{t}(D_t)_{ii} - (n-1)\psum \Big| \geq \delta' (n-1)\psum \right) \leq 2\exp\Big(-\frac{\delta'^2}{3} (n-1)\psum \Big),
\end{align*}
which implies via a union bound that 
\begin{align*}
    \prob\Big(\Big\|\sum_t D_t - (n-1)\psum \Big\|_2 \geq \delta' (n-1)\psum \Big) \leq 2n \exp\Big(-\frac{\delta'^2}{3} (n-1)\psum \Big).
\end{align*}
Choosing $\delta' = \sqrt{\frac{3\log(2n/\delta)}{(n-1)\psum}}$ and assuming $\psum \geq c_1\frac{\log(n/\delta)}{n}$ for a suitably large constant $c_1 > 0$, it follows that there exists $c_2 > 0$ such that with probability at least $1-\delta$,
\begin{align} \label{eq:D_conc_transync}
    \norm{\sum_t(D_t - \Dbar_t)}_2 \leq c_2 \sqrt{n \log(n/\delta) \psum}.
\end{align}
\paragraph{Bounding $\norm{\sum_t(A_t - \Abar_t)}_2 $.}
This is achieved by a direct application of \cite[Theorem 1]{paul_ml_sbm20}, which in our setup yields
\begin{align*}
    \prob\left(\norm{\sum_t(A_t - \Abar_t)}_2 \leq \sqrt{(n-1)\psum \log(2n/\delta)}\right) \geq 1-\delta
\end{align*}
if $(n-1)\psum \geq (4/9) \log(2n/\delta)$. In simplified terms, this means that there exist constants $c_!, c_2 > 0$ such that if $\psum \geq c_1 \frac{\log (n/\delta)}{n}$, then 
\begin{align} \label{eq:A_conc_transync}
    \prob\left(\norm{\sum_t(A_t - \Abar_t)}_2 \leq c_2 \sqrt{n\psum \log(n/\delta)}\right) \geq 1-\delta.
\end{align}
Applying \eqref{eq:A_conc_transync}, \eqref{eq:D_conc_transync} in \eqref{eq:L_conc_transync} leads to the bound
\begin{equation} \label{eq:l_conc_fin_transync}
    \prob\left(\norm{\sum_t(L_t - \Lbar_t)}_2 \leq c_3 \sqrt{n\psum \log(n/\delta)} \right) \geq 1-2\delta.
\end{equation}
Since $\sum_t \Lbar = n\psum(I_n - \frac{\ones_n \ones_n^\top}{n})$, we obtain the statement of part 1 using \eqref{eq:l_conc_fin_transync} and Weyl's inequality \cite{Weyl1912}, provided $\psum \geq c\frac{\log(n/\delta)}{n}$ for a large enough constant $c > 0$.

\subsubsection{Proof of part 2}
We start by noting that 
\begin{align} \label{eq:normC_bd_transync}
    \|C\|_2 = \max_{t \in [T]} \sqrt{\lambmax(L_t)} \leq \max_{\stackrel{t\in [T]}{i \in [n]}} \sqrt{2\sum_{j\neq i} (A_t)_{ij}}
\end{align}
where the inequality follows from Gershgorin's disk theorem (e.g. \cite{Horn_Johnson_1985}). Now applying Hoeffdings inequality (e.g., \cite{boucheron_book}) and denoting $\pmax := \max_t p_t$, we obtain via a union bound for any $b > 0$, 
\begin{align*}
    \prob\left(\max_{\stackrel{t\in [T]}{i \in [n]}} \sum_{j\neq i} (A_t)_{ij} \geq (n-1) \pmax + b \right) \leq nT e^{-\frac{2b^2}{n}}.
\end{align*}
Plugging this in \eqref{eq:normC_bd_transync} for the choice $b = \sqrt{\frac{n}{2} \log(\frac{nT}{\delta})}$, and observing from \eqref{eq:normC_bd_transync} that $\norm{C_2} \leq \sqrt{2n}$ a.s., we obtain the statement of part 2 after some minor simplifications.